\documentclass{amsart}

\usepackage{amsmath}
\usepackage{amsthm}
\usepackage{amsfonts}
\usepackage{amssymb}
\usepackage{graphics}

\newcommand\R{{\mathbf{R}}}
\newcommand\C{{\mathbf{C}}}
\newcommand\Z{{\mathbf{Z}}}

\renewcommand\P{{\mathbf{P}}}
\newcommand\E{{\mathbf{E}}}

\newcommand\eps{{\varepsilon}}

\newcommand\tr{\operatorname{trace}}

\numberwithin{equation}{section} 
\theoremstyle{plain}
\newtheorem{theorem}[subsection]{Theorem}
\newtheorem{proposition}[subsection]{Proposition}
\newtheorem{lemma}[subsection]{Lemma}
\newtheorem{corollary}[subsection]{Corollary}
\newtheorem{definition}[subsection]{Definition}

\newtheorem{remark}[subsection]{Remark}

\begin{document}

\title{Mixing for progressions in non-abelian groups}

\author{Terence Tao}
\address{UCLA Department of Mathematics, Los Angeles, CA 90095-1555.}
\email{tao@math.ucla.edu}


\begin{abstract}
We study the mixing properties of progressions $(x,xg,xg^2)$, $(x,xg,xg^2,xg^3)$ of length three and four in a model class of finite non-abelian groups, namely the special linear groups $\operatorname{SL}_d(F)$ over a finite field $F$, with $d$ bounded.  For length three progressions $(x,xg,xg^2)$, we establish a strong mixing property (with error term that decays polynomially in the order $|F|$ of $F$), which among other things counts the number of such progressions in any given dense subset $A$ of $\operatorname{SL}_d(F)$, answering a question of Gowers for this class of groups.  For length four progressions $(x,xg,xg^2,xg^3)$, we establish a partial result in the $d=2$ case if the shift $g$ is restricted to be diagonalisable over $F$, although in this case we do not recover polynomial bounds in the error term. Our methods include the use of the Cauchy-Schwarz inequality, the abelian Fourier transform, the Lang-Weil bound for the number of points in an algebraic variety over a finite field, some algebraic geometry, and (in the case of length four progressions) the multidimensional Szemer\'edi theorem.
\end{abstract}

\maketitle

\section{Introduction}

Let $G = (G,\cdot)$ be a finite group, not necessarily abelian.  
Given a natural number $k \geq 1$ and $k$ functions $f_0, \dots, f_{k-1} \colon G \to \C$, we define the $k$-linear form
$$ \Lambda_{k,G}(f_0,\dots,f_{k-1}) := \E_{x,g \in G} \prod_{i=0}^{k-1} f_i(xg^{i-1}),$$
where $\E$ denotes the averaging notation
$$ \E_E f := \E_{x\in E} f(x) := \frac{1}{|E|} \sum_{x \in E} f(x)$$
for non-empty finite sets $E$ and complex-valued functions $f$ on $E$, with $|E|$ denoting the cardinality of the set $E$.  
Thus, for instance, if $A$ is a subset of $G$, with the associated indicator function $1_A: G \to \{0,1\}$, $\Lambda_{k,G}(1_A,\dots,1_A)$ denotes the number of (possibly degenerate) length $k$ geometric progressions $(x,xg,\dots,xg^{k-1})$ in $A$, divided by $|G|^k$.  

The form $\Lambda_{k,G}$ is easily computed for $k=1,2$:
\begin{align*}
\Lambda_{1,G}(f_0) &= \E_G f_0 \\
\Lambda_{2,G}(f_0,f_1) &= (\E_G f_0) (\E_G f_1).
\end{align*}

Now we turn to the $k=3$ case.  If $f_0,f_1,f_2$ are selected in a sufficiently ``random'' fashion, then probabilistic heuristics suggest that one has
\begin{equation}\label{heuristic}
\Lambda_{3,G}(f_0,f_1,f_2) \approx (\E_G f_0) (\E_G f_1) (\E_G f_2)
\end{equation}
and more generally
\begin{equation}\label{heuristic-k}
\Lambda_{k,G}(f_0,\dots,f_{k-1}) \approx \prod_{i=0}^{k-1} \E_G f_i.
\end{equation}
However, if $G$ has a non-trivial low-dimensional unitary representation $\rho: G \to U_d(\C)$ for some small $d$, then it becomes possible to violate the heuristic \eqref{heuristic}.  Indeed, if one lets $B$ be a small neighbourhood of the identity in $U_d(\C)$, and sets $B'$ to be the slightly larger neighbourhood
$$ B' := B \cdot B^{-1} \cdot B := \{ b_1 b_2^{-1} b_3: b_1,b_2,b_3 \in B\},$$
with the associated preimages $A := \rho^{-1}(B), A' := \rho^{-1}(B')$, then from the identity
$$ \rho(xg^2) = \rho(xg) \rho(x)^{-1} \rho(xg)$$
we see that $xg^2 \in B'$ whenever $x,xg \in B$.  In particular, we have
$$ \Lambda_{3,G}(1_A, 1_A, 1_{A'}) = \Lambda_{2,G}(1_A, 1_A) = (\E_G 1_A) (\E_G 1_A),$$
which violates \eqref{heuristic} if $B$ (and hence $B'$) is small enough; if the dimension $d$ is small, this can be done with a relatively large value for the density $\E_G 1_A$.  A similar argument applies to exhibit a deviation from \eqref{heuristic-k} for any $k \geq 3$.

The deviation from \eqref{heuristic} is most pronounced in the case when $G$ is abelian (so that all irreducible unitary representations of $G$ are in fact one-dimensional).  In this case we will switch to additive notation and write the group operation of $G$ as $+$, so that
\begin{equation}\label{xgo}
 \Lambda_{3,G}(f_0,f_1,f_2) := \E_{x,g \in G} f_0(x) f_1(x+g) f_2(x+2g).
\end{equation}
The analysis of this form usually begins by introducing the Fourier transform
$$ \hat f(\xi) := \E_{x \in G} f(x) e(-\xi \cdot x)$$
for all $\xi$ in the Pontryagin dual $\hat G$ of $G$, defined as the space of all homomorphisms $\xi: x \mapsto \xi\cdot x$ from $G$ to the (additive) unit circle $\R/\Z$, where $e(x) :=e^{2\pi ix}$; of course, $\hat G$ is encoding the irreducible one-dimensional unitary representations of $G$ mentioned previously.  Using the Fourier inversion formula
$$f(x) := \sum_{\xi \in \hat G} \hat f(\xi) e(\xi \cdot x)$$
one soon arrives at the useful identity
$$ \Lambda_{3,G}(f_0,f_1,f_2) = \sum_{\xi \in \hat G} \hat f_0(\xi) \hat f_1(-2\xi) \hat f_2(\xi)$$
relating the magnitude of $\Lambda_{3,G}(f_0,f_1,f_2)$ with the size of the Fourier coefficients of $f_0,f_1,f_2$.  Note that the heuristic \eqref{heuristic} corresponds to the $\xi=0$ term in this sum; the point is that the non-zero frequencies $\xi \neq 0$ can also give a significant contribution.

Using the above identity, one can eventually establish the Roth-type theorem
\begin{equation}\label{laga}
 \Lambda_{3,G}(1_A,1_A,1_A) \ge c_3(\delta)
\end{equation}
for any $0 < \delta \leq 1$, any finite abelian group $G$, and any subset $A \subset G$ with $|A| \geq \delta |G|$, where $c_3(\delta)>0$ depends only on $\delta$; see e.g. \cite[Theorem 10.9]{tao-vu}.   In a similar vein, we have the deep theorem of Szemer\'edi \cite{szemeredi}, which implies\footnote{Strictly speaking, the original theorem of Szemer\'edi only treats the case when $G$ is a cyclic group, but subsequent proofs of Szemer\'edi's theorem (such as the hypergraph-based proofs in \cite{gowers-hyper}, \cite{rs}, \cite{rodl}, \cite{tao-hyper}) allow for one to handle arbitrary abelian groups $G$.} the more general lower bound,
\begin{equation}\label{laga-k}
 \Lambda_{k,G}(1_A,\dots,1_A) \ge c_k(\delta)
\end{equation}
for all $k \geq 1$ and $0 < \delta \leq 1$, any finite abelian groups $G$, and any $A \subset G$ with $|A| \geq \delta |G|$, where $c_k(\delta)>0$ depends only on $k$ and $\delta$.

\begin{remark}  More explicit bounds for $c_3(\delta)$ are known.  For general abelian groups $G$, an argument of Bourgain \cite{bourgain-triples} gives $c_3(\delta) \geq c \delta^{C/\delta^2}$ for some absolute constants $c,C>0$; see e.g. \cite[Theorem 10.30]{tao-vu}.  In the case when $G$ is a cyclic group, the strongest bound to date is due to Sanders \cite{sanders}, who (in our notation) established that $c_3(\delta) \geq c \delta^{C \log^4(1/\delta)/\delta}$; on the other hand, in this case one also has the upper bound $c_3(\delta) \leq C \delta^{c \log(1/\delta)}$ due to Behrend \cite{behrend}.  When $G$ is a vector space over a fixed finite field $F$ of odd order (such as $F_3$), the best bound is due to Bateman and Katz \cite{bate}, who established $c_3(\delta) \geq \exp(-C \delta^{c-1})$ for some constants $C,c>0$ depending only on $F$.  For $k > 3$ and for cyclic groups, the explicit bounds known are weaker: for $k=4$, the results in \cite{gt} give $c_4(\delta) \geq c \exp(- C \delta^{-C \log(1/\delta)})$, while for higher $k$, the results in \cite{gow} give $c_k(\delta) \geq c_k \exp(\exp(-C_k \delta^{-C_k}))$ for some constants $c_k, C_k>0$ depending on $k$; in the other direction, a modification of the Behrend construction \cite{rankin} gives $c_k(\delta) \leq C_k \delta^{c_k \log^{c_k}(1/\delta)}$.  For general groups, explicit lower bounds on $c_k(\delta)$ are known thanks to the recent quantitative work on the density Hales-Jewett theorem \cite{polymath} or the hypergraph removal lemma \cite{gowers}, \cite{rs}, \cite{rodl}, \cite{tao-hyper}, but the bounds are rather poor.
\end{remark}

Now we turn to the case when $G$ is not necessarily abelian, and in particular in the \emph{quasirandom} case in which $G$ has no low-dimensional representations.  More precisely, following Gowers \cite{gowers}, call a finite group $G$ \emph{$D$-quasirandom} if the only irreducible unitary representations $\rho: G \to U_d(\C)$ have dimension $d$ greater than or equal to $D$.  A model example of quasirandom groups are provided by the special linear groups over a finite field:

\begin{proposition}[Quasirandomness of special linear group]\label{quasi-special}  Let $d \geq 2$ be an integer, and let $F$ be a finite field.  Then the group $\operatorname{SL}_d(F)$ of $d \times d$ matrices with coefficients in $F$ of determinant one is $c_d |F|^{d-1}$-quasirandom, for some $c_d>0$ depending only on $d$.
\end{proposition}

\begin{proof} This follows from the results in \cite{land}.  The case when $d=2$ and $|F|$ has prime order is classical, dating back to the work of Frobenius.  Similar results hold for other finite (almost) simple groups of Lie type and bounded rank; see \cite{land}.
\end{proof}

When $D$ is large, one expects better mixing properties in the forms $\Lambda_{k,G}$.  To illustrate this, we introduce the variant expressions
$$ \Lambda_{k,G}^*(f_0,\dots,f_{k-1}) := \E_{g \in G} \left|\E_{x \in G} \prod_{i=0}^{k-1} f_i(xg^{i-1}) - \prod_{i=0}^{k-1} \E_G f_i\right|,$$
which controls the number of length $k$ progressions for a \emph{single} (generic) shift $g$, as opposed to the average number over all such $g$.
This expression vanishes for $k=1$, but can be non-trivial for $k>1$.
From the triangle inequality we have
\begin{equation}\label{tri-eq}
 |\Lambda_{k,G}(f_0,\dots,f_{k-1}) - \prod_{i=0}^{k-1} \E_G f_i| \leq \Lambda_{k,G}^*(f_0,\dots,f_{k-1})
 \end{equation}
and so the heuristic \eqref{heuristic-k} holds whenever $\Lambda_{k,G}^*(f_0,\dots,f_{k-1})$ is small.  However, when one has a low-dimensional representation $\rho: G \to U_d(\C)$, it is possible for $\Lambda_{k,G}^*(f_0,\dots,f_{k-1})$ to be large even when \eqref{heuristic-k} holds.  Consider for instance the $k=2$ case, in which \eqref{heuristic-k} holds exactly.  If we let $B$ be a small neighbourhood of the identity in $U_d(\C)$ with preimage $A := \rho^{-1}(B)$ as before, and sets $A' := \rho^{-1}(B^{-1} \cdot B)$, we see that $1_A(x) 1_A(xg)$ vanishes whenever $g \not \in A'$, and thus
$$ \Lambda_{2,G}^*(1_A,1_A) = \E_{g \in G} |\E_{x \in G} 1_A(x) 1_A(xg) - (\E_G 1_A)^2|$$
can be lower bounded by $(\E_G 1_A)^2 (1 - \E_G 1_{A'})$, which can be somewhat large if $B$ is chosen small enough, and $d$ is small.

As observed first by Gowers \cite{gowers}, though, $\Lambda_{2,G}^*$ becomes much smaller in the quasirandom case.  This is elegantly captured by the inequality
\begin{equation}\label{babai}
\| f_1 * f_2 \|_{L^2(G)} \leq D^{-1/2} |G| \|f_1\|_{L^2(G)} \|f_2\|_{L^2(G)}
\end{equation}
of Babai, Nikolov, and Pyber \cite{babai-nikolov-pyber}, for any $D$-quasirandom group $G$ and any functions $f_1,f_2: G \to \C$ with at least one of $f_1,f_2$ having mean zero, where
$$ \|f\|_{L^2(G)} := (\E_{x \in G} |f(x)|^2)^{1/2}$$
and $*$ denotes the discrete\footnote{Ordinarily, one would normalise this convolution by $\frac{1}{|G|}$ for compatibility with the averaging in the $L^2(G)$ norm, but it will be convenient to use the discrete normalisation because we will be passing from a group $G$ to various subgroups of $G$ in subsequent arguments.} convolution
$$ f_1*f_2(x) := \sum_{y \in G} f_1(y) f_2(y^{-1} x) = \sum_{y \in G} f_1(xy^{-1}) f_2(y);$$
see \cite{babai-nikolov-pyber} or \cite[Proposition 3]{berg}.  Note that \eqref{babai} improves by a factor of $D^{-1/2}$ over the trivial bound of $|G| \|f_1\|_{L^2(G)} \|f_2\|_{L^2(G)}$ arising from the Young and Cauchy-Schwarz inequalities.

The estimate \eqref{babai} has the following useful corollary:

\begin{lemma}[$k=2$ mixing for quasirandom groups]\label{quasmix}  If $G$ is a $D$-quasirandom group, then
$$ \Lambda_{2,G}^*(f_1,f_2) \leq D^{-1/2} \|f_1\|_{L^2(G)} \|f_2\|_{L^2(G)}.$$
\end{lemma}

\begin{proof}  Observe that the expression $\Lambda_{2,G}^*(f_1,f_2)$ does not change if $f_1$ or $f_2$ is modified by an additive constant.  Thus we may normalise $f_1$ and $f_2$ to both have mean zero.  We can then write
$$ \Lambda_{2,G}^*(f_1,f_2) = \E_{g \in G} f_0(g) \E_{x \in G} f_1(x) f_2(xg)$$
for some function $f_0: G \to \C$ of magnitude $1$.  The right-hand side can be rewritten, after a change of variables, as
$$ \frac{1}{|G|} \E_{y \in G} (f_0*f_1)(y) f_2(y).$$
The claim then follows from \eqref{babai} and the Cauchy-Schwarz inequality.
\end{proof}

In \cite{gowers}, Gowers posed the question of whether results such as Lemma \ref{quasmix} could be extended to higher values of $k$, so that the heuristic \eqref{heuristic} or \eqref{heuristic-k} could hold for sufficiently quasirandom groups.  We were not able to settle this question in general, but in the $k=3$ case we can affirmatively answer the question for a model class of quasirandom groups, namely the special linear groups $\operatorname{SL}_d(F)$ over a finite field $F$:

\begin{theorem}\label{first}  Let $F$ be a finite field, and set $G := \operatorname{SL}_d(F)$ for some $d \geq 2$.  Then we have
$$
|\Lambda_{3,G}^*(f_0,f_1,f_2)| \ll_d |F|^{-\min(d-1,2)/8} \prod_{i=0}^2 \|f_i\|_{L^\infty(G)} 
$$
for all functions $f_0,f_1,f_2: G \to \C$, where $\|f\|_{L^\infty(G)} := \sup_{x \in G} |f(x)|$.  
Here and in the sequel we use $Y \ll_d X$, $X \gg_d Y$, or $Y = O_d(X)$ to denote the estimate $|Y| \leq C_d X$ for some $C_d$ depending only on $d$, and similarly with $d$ replaced by other sets of parameters.  In particular, from \eqref{tri-eq} one has 
$$
 \Lambda_{3,G}(f_0,f_1,f_2) = (\E_G f_0) (\E_G f_1) (\E_G f_2) + O_d( |F|^{-\min(d-1,2)/8} \prod_{i=0}^2 \|f_i\|_{L^\infty(G)}  )
$$
\end{theorem}

Theorem \ref{first} is proven primarily through application of the Cauchy-Schwarz inequality and Lemma \ref{quasmix}; we give this proof in Sections \ref{3g-sec}-\ref{sld-sec}.  The key point is that the non-abelian nature of $G$ means that the application of Cauchy-Schwarz creates more averaging than is seen in the abelian case.  The exponent $\min(d-1,2)/8$ is unlikely to be optimal.  By taking $f_0,f_1,f_2$ to be constant on left cosets $gH$ of a proper subgroup of $H$ and of mean zero, we see that one cannot replace the quantity $|F|^{-\min(d-1,2)/8}$ by anything much smaller than $|H|/|G|$; in particular, if we take $H$ to be the Borel subgroup of upper-triangular matrices in $G$, we see that one cannot replace $\min(d-1,2)/8$ by any exponent greater than $\frac{d(d-1)}{2}$.  It is likely that one can extend Theorem \ref{first} to other finite simple groups\footnote{To be pedantic, $\operatorname{SL}_d(F)$ is usually not a simple group, due to its nontrivial centre; but it is a bounded cover of a finite simple group, namely $P\operatorname{SL}_d(F)$.  Note that the results for $\operatorname{SL}_d(F)$ in this paper automatically descend to the quotient group $P\operatorname{SL}_d(F)$ without difficulty.} of Lie type with bounded rank, but we will not do so here.

Applying Theorem \ref{first} to indicator functions $f_0=1_A, f_1 = 1_B, f_2 = 1_C$ and using Markov's inequality, we conclude in particular the ``weak mixing''  bound
$$ \mu( A \cap Bg \cap Cg^2 ) = \mu(A) \mu(B) \mu(C) + O_d(|F|^{-\min(d-1,2)/16})$$
for a proportion $1-O_d(|F|^{-\min(d-1,2)/16})$ of $g \in G$, where $\mu(A) := \E_G 1_A = |A|/|G|$ denotes the density of $A$ in $G$.

We conjecture that Theorem \ref{first} can be extended to higher values of $k$ than $k=3$ (possibly with a smaller exponent than $\min(d-1,2)/8$).  Unfortunately, the Cauchy-Schwarz argument does not seem to extend beyond $k=3$; in contrast to the abelian case, in the non-abelian setting it appears that when $k>3$, each application of Cauchy-Schwarz \emph{increases} the complexity of the resulting form, rather than decreasing it as in the abelian case.  However, we are able to establish the following weak partial result in the $k=4$, $d=2$ case, in which the shift $g$ is restricted to be diagonalisable:

\begin{theorem}\label{second} Let $F$ be a finite field, and set $G := \operatorname{SL}_2(F)$.  Let $S$ denote all the elements of $G$ which are diagonalisable over $F$.  Then for all functions $f_0,f_1,f_2,f_3: G \to \C$, one has
$$
\E_{g \in S} \left|\E_{x \in G} \prod_{i=0}^{3} f_i(xg^{i-1}) - \prod_{i=0}^{k-1} \E_G f_i\right| = o_{|F| \to \infty}( \prod_{i=0}^{3} \|f_i\|_{L^\infty(G)} ),$$
where $o_{|F| \to \infty}(X)$ denotes a quantity bounded by $c(|F|) X$ for some quantity $c(|F|)$ that goes to zero as $|F|$ goes to infinity.
\end{theorem}

It is easy to show that for large $|F|$, $S$ has density about $1/2$ in $G$; see Section \ref{Reduction-sec}.  The main reason why the shift $g$ is restricted to $S$ in our arguments is in order to ensure that $g$ is contained in a non-trivial metabelian subgroup of $G$; for instance, if $g$ is a diagonal matrix with entries in $F$, then it is contained in the Borel subgroup $B$ of upper triangular matrices in $G$.  The argument is rather \emph{ad hoc} in nature, combining Cauchy-Schwarz and the abelian Fourier transform with some explicit nonabelian effects coming from the algebraic structure of progressions in the Borel group.  It also relies on (a quantitative version of) the multidimensional Szemer\'edi theorem of Furstenberg and Katznelson \cite{fk0}, which is the reason for the poor decay in $|F|$.  Finally, to pass from the Borel subgroup back to the full group, an expansion result in $\operatorname{SL}_2(F)$, related to the Bourgain-Gamburd expansion theory in this group, is also required.

\begin{remark}  The results in this paper concern the mixing properties of the patterns $(x,xg,xg^2)$ and $(x,xg,xg^2,xg^3)$ for an explicit class of quasirandom groups, namely the special linear groups.  In a recent paper with Vitaly Bergelson \cite{berg}, we also establish some mixing properties for the patterns $(x,xg,gx)$ and $(g,x,xg,gx)$ in arbitrary quasirandom groups.  While the end results of both papers are superficially similar in nature, the proof techniques turn out to be completely different, with the results in \cite{berg} relying on nonstandard analysis, the triangle removal lemma from graph theory, and ergodic theorems involving idempotent ultrafilters.  In both cases, the methods are tailored to the specific patterns being counted, and it appears we are still quite far from a general theory that can cover all nonabelian patterns involving two or more variables such as $x,g$.  

We also remark that in \cite{tao-etale}, some mixing properties of patterns of the form $(x,y,P(x,y))$ were established when $P: G \times G \to G$ was a definable function over a finite field of large characteristic.  However, the arguments in that paper (which also involve Cauchy-Schwarz, but applied in a slightly different fashion) required $\{ (P(x,y), P(x,y'), P(x',y), P(x',y')): x,y \in G \}$ to be sufficiently Zariski dense in $G^4$.  This is not the case for the pattern $(x,xg,xg^2)$ (in which $P(x,y) := yx^{-1}y$), since $P(x,y)$ and $P(x,y')$ are necessarily conjugate to each other.
\end{remark}

\subsection{Acknowledgments}

The author was partially supported by a Simons Investigator award from the Simons Foundation and by NSF grant DMS-0649473.  He also thanks Vitaly Bergelson for many stimulating discussions regarding these topics.

\section{A general bound for $\Lambda_{3,G}$}\label{3g-sec}

Let us define the \emph{reduced spectral norm} $\|\mu\|_{S(G)}$ of a function $\mu: G \to \C$ to be the best constant such that
\begin{equation}\label{reduced}
\| f * \mu \|_{L^2(G)} \leq \|\mu\|_{S(G)} \|f\|_{L^2(G)}
\end{equation}
whenever $f: G \to \C$ has mean zero, thus
\begin{equation}\label{mox}
 |\E_{z \in G} f_1(z) (f_2*\mu)(z)| \leq \|\mu\|_{S(G)} \|f_1\|_{L^2(G)} \|f_2\|_{L^2(G)}
\end{equation}
for all $f_1,f_2: G \to \C$, as can be seen by splitting $f_1,f_2$ into constant and mean zero components, and noting that all cross terms vanish.

\begin{remark}\label{pwt}  From the Peter-Weyl theorem, one can also write $\|\mu\|_{S(G)}$ as
$$ \| \mu\|_{S(G)} = \sup_\rho \| \sum_{g \in G} \mu(g) \rho(g) \|_{\operatorname{op}}$$
where $\rho: G \to U(V)$ ranges over all non-trivial irreducible finite-dimensional unitary representations of $G$.  We will not make much use of this representation-theoretic interpretation of the reduced spectral norm here, although we remark that this interpretation can be used to derive the basic quasirandomness inequality \eqref{babai} (or \eqref{mus-2} below).
\end{remark}

The reduced spectral norm $\|\mu\|_{S(G)}$ is clearly a seminorm, and in particular obeys the triangle inequality.  From Minkowski's inequality, we have the crude bound
\begin{equation}\label{mus}
 \|\mu\|_{S(G)} \leq \|\mu\|_{\ell^1(G)}.
\end{equation}
From \eqref{babai} we also have the more refined estimate
\begin{equation}\label{mus-2}
 \|\mu\|_{S(G)} \leq D^{-1/2} |G|^{1/2} \|\mu\|_{\ell^2(G)}
 \end{equation}
when $G$ is $D$-quasirandom.  If we split $\mu$ into the region where $\mu(x) > C_0/|G|$, and the region where $\mu(x) \leq C_0/|G|$, for some threshold $C_0>0$, and apply \eqref{mus} to the latter and \eqref{mus-2} to the former, we conclude that
\begin{equation}\label{mus-3}
 \|\mu\|_{S(G)} \leq C_0 D^{-1/2} + \sum_{x \in G: \mu(x) > C_0/|G|} \mu(x).
\end{equation}

By combining these estimates with the Cauchy-Schwarz inequality, we can obtain the following general bound on the quantity $\Lambda_3(f_0,f_1,f_2)$.

\begin{proposition}\label{gen}  Let $G=(G,\cdot)$ be a $D$-quasirandom group for some $D \geq 1$. Let $C_0 \geq 1$ be a parameter.  Then we have
\begin{equation}\label{lag-0}
\Lambda^*_{3,G}(f_0,f_1,f_2) \ll \left(C_0 D^{-1/2} + \E_{b,h \in G} \sum_{y \in G: \mu_{b,h}(y) \geq C_0/|G|} \mu_{b,h}(y)\right)^{1/4} \prod_{i=0}^2 \|f_i\|_{L^\infty(G)} 
\end{equation}
for all functions $f_0,f_1,f_2: G \to \C$, where for each $b,h \in G$, $\mu_{b,h}: G \to \C$ is the function
\begin{equation}\label{mobuh}
 \mu_{b,h} := \E_{g \in G} \E_{c \in Z(b)} \delta_{g c^{-1} h^{-1} g^{-1} c^{-1} h^{-1}}
\end{equation}
where $Z(b) := \{ c \in G: cb=bc\}$ is the centraliser of $b$.  
\end{proposition}

One can view $\mu_{b,h}$ as a probability measure on $G$, describing the distribution of the random variable $g c^{-1} h^{-1} g^{-1} c^{-1} h^{-1}$ when $g$ is a randomly chosen element of $G$, and $c$ is a random element commuting with $b$.  The estimate \eqref{lag-0} becomes useful when $\mu_{b,h}$ is approximately uniformly distributed over $G$ for typical $b,h$, so that $\sum_{y \in G: \mu_{b,h}(y) \geq C_0/|G|} \mu_{b,h}(y)$ is small.

\begin{proof}  When $f_0$ is equal to a constant $c$, we have
$$ \Lambda^*_{3,G}(f_0,f_1,f_2) = |c| \Lambda^*_{2,G}(f_1,f_2)$$
and the claim then follows from Lemma \ref{quasmix}.  As $\Lambda^*_{3,G}$ is sublinear in each of the three arguments, we may thus assume that $f_0$ has mean zero.  We then also assume that $f_0,f_1,f_2$ are real-valued, and normalise so that
$$\|f_0\|_{L^\infty(G)} = \|f_1\|_{L^\infty(G)} = \|f_2\|_{L^\infty(G)} = 1.$$
Our task is now to show that
$$
|\Lambda^*_{3,G}(f_0,f_1,f_2)|^4 \ll C_0 D^{-1/2} + \E_{b,h \in G} \sum_{y \in G: \mu_{b,h}(y) \geq C_0/|G|} \mu_{b,h}(y).$$
Ever since the work of Gowers \cite{gowers-4aps}, it has been is common to control expressions such as $\Lambda^*_{3,G}(f_0,f_1,f_2)$ via the Cauchy-Schwarz inequality.  In the literature, this was mostly performed in the abelian case, but one can obtain a useful estimate via Cauchy-Schwarz in the non-abelian case too.  First, we shift $x$ by $g^{-1}$ to obtain
$$ \Lambda^*_{3,G}(f_0,f_1,f_2) = \E_{g \in G} |\E_{x \in G} f_0(xg^{-1}) f_1(x) f_2(xg)|$$
which we expand as
$$ \Lambda^*_{3,G}(f_0,f_1,f_2) = \E_{x \in G} f_1(x)  (\E_{g \in G} f_0(xg^{-1}) f_2(xg) f_3(g))$$
for some\footnote{If one is only interested in bounding $\Lambda_{3,G}(f_0,f_1,f_2)$ rather than $\Lambda_{3,G}^*(f_0,f_1,f_2)$, one can take $f_3 \equiv 1$, and the reader may wish to do so initially in the argument that follows in order to simplify the exposition.} function $f_3: G \to \C$ bounded in magnitude by $1$.  Applying Cauchy-Schwarz in $x$ to eliminate $f_1$, we obtain
$$ \Lambda^*_{3,G}(f_0,f_1,f_2) \leq (\E_{x \in G} |\E_{g \in G} f_0(xg^{-1}) f_2(xg) f_3(g)|^2)^{1/2}.$$
We can expand the right-hand side as
$$ (\E_{x,g,g' \in G} f_0(xg^{-1}) f_0(x(g')^{-1}) f_2(xg) f_2(xg') f_3(g) f_3(g'))^{1/2}.$$
Making the change of variables $(y,g,a) := (xg, g, g^{-1} g')$, this becomes
$$ (\E_{y,g,a \in G} f_0(yg^{-2}) f_0(yg^{-1}a^{-1}g^{-1}) f_2(y) f_2(ya) f_3(g) f_3(ga))^{1/2}.$$
If we define $\Delta_a f(y) := f(y) f(ya)$, this becomes
$$ \left(\E_{y,a \in G} \Delta_a f_2(y) (\E_{g \in G} \Delta_{g a^{-1} g^{-1}} f_0(yg^{-2}) \Delta_a f_3(g))\right)^{1/2}.$$
Applying Cauchy-Schwarz in $y,a$ to eliminate $\Delta_a f_2$, we thus have
$$ \Lambda_{3,G}^*(f_0,f_1,f_2) \leq
(\E_{y,a \in G} |\E_{g \in G} \Delta_{g a^{-1} g^{-1}} f_0(yg^{-2}) \Delta_a f_3(g)|^2)^{1/4}.$$
The right-hand side can be expanded as
$$
(\E_{y,a,g,g' \in G} \Delta_{g a^{-1} g^{-1}} f_0(yg^{-2}) \Delta_{g' a^{-1} (g')^{-1}} f_0(y(g')^{-2}) \Delta_a f_3(g) \Delta_a f_3(g'))^{1/4}.$$
Making the change of variables $(z,b,g,h) := (yg^{-2}, ga^{-1}g^{-1}, g, g' g^{-1})$, we conclude the inequality
\begin{equation}\label{zag}
|\Lambda_{3,G}(f_0,f_1,f_2)| \leq (\E_{z, b, g, h \in G} \Delta_b f_0(z) \Delta_{hbh^{-1}} f_0( z g h^{-1} g^{-1} h^{-1} )
\Delta_{g^{-1} b^{-1} g} f_3(g) \Delta_{g^{-1} b^{-1} g} f_3(hg))^{1/4}.
\end{equation}
The right-hand side of \eqref{zag} can be viewed as a twisted, weighted variant\footnote{Indeed, in the model case when $f_3 \equiv 1$ and $G$ is abelian, the right hand side simplifies to $(\E_{z,b,h \in G} \Delta_b f_0(z) \Delta_b f_0(z h^{-2}))^{1/4}$, which (in the case that $G$ has odd order) is precisely the Gowers norm $\|f_0\|_{U^2(G)}$.} of the Gowers $U^2$ norm \cite{gowers-4aps}.   To control it, we observe the self-averaging identity
$$ \E_{h \in G} F(h) = \E_{h \in G} \E_{c \in C} F(hc)$$
for any non-empty set $C$ and any function $F: G \to \C$.  We apply this identity with $C$ equal to the centraliser $Z(b) := \{ c \in G: cb=bc\}$ of $b$ and $F$ equal to the expression being averaged on the right-hand side of \eqref{zag}; the point of this averaging is to exploit the trivial observation that the function $\Delta_{hbh^{-1}} f_0$ does not change if one replaces $h$ by $hc$ for an arbitrary $c \in Z(b)$.  We conclude that
\begin{align*}
|\Lambda_{3,G}(f_0,f_1,f_2)| &\leq (\E_{z, b, g, h \in G} \E_{c \in Z(b)} \Delta_b f_0(z) \Delta_{hbh^{-1}} f_0( z g c^{-1} h^{-1} g^{-1} c^{-1} h^{-1} ) \\
&\quad \Delta_{g^{-1} b^{-1} g} f_3(g) \Delta_{g^{-1} b^{-1} g} f_3(hcg))^{1/4}.
\end{align*}
We can rewrite the right-hand side as
\begin{equation}\label{noo}
 |\E_{b,h \in G} \E_{z \in G} \Delta_b f_0(z) (\Delta_{hbh^{-1}} f_0 * \tilde \mu_{b,h})(z)|^{1/4}
\end{equation}
where $\tilde \mu_{b,h}$ is a weighted version\footnote{Returning to the model case when $f_3 \equiv 1$ and $G$ is an abelian group of odd order, we have in this case that $\tilde \mu_{b,h} \equiv 1/|G|$, and \eqref{noo} is again just the Gowers norm $\|f_0\|_{U^2(G)}$.  The point is that for certain non-abelian groups $G$, one can still obtain some sort of equidistribution control on $\tilde \mu_{b,h}$ that makes it behave roughly like the uniform distribution $1/|G|$.} of $\mu_{b,h}$:
$$ \tilde \mu_{b,h} := \E_{g \in G} \E_{c \in Z(b)} \delta_{g c^{-1} h^{-1} g^{-1} c^{-1} h^{-1}} \Delta_{g^{-1} b^{-1} g} f_3(g) \Delta_{g^{-1} b^{-1} g} f_3(hcg).$$
Our task is now to show that
\begin{equation}\label{zog}
 |\E_{b,h \in G} \E_{z \in G} \Delta_b f_0(z) (\Delta_{hbh^{-1}} f_0 * \tilde \mu_{b,h})(z)| \ll
C_0 D^{-1/2} + \E_{b,h \in G} \sum_{y \in G: \mu_{b,h}(y) \geq C_0/|G|} \mu_{b,h}(y).
\end{equation}

From \eqref{mox} we see that
$$ |\E_{z \in G} \Delta_b f_0(z) (\Delta_{hbh^{-1}} f_0 * \tilde \mu_{b,h})(z)| \leq
\| \tilde \mu_{b,h} \|_{S(G)} + |\E_{z \in G} \Delta_b f_0(z)|$$
(by splitting $\Delta_b f_0$ into constant and mean zero components).  We may thus upper bound the left-hand side of \eqref{zog} by
$$
\E_{b,h \in G} \| \tilde \mu_{b,h} \|_{S(G)} +  \E_{b \in G} |\E_{z \in G} \Delta_b f_0(z)|.$$
The second term is equal to $\Lambda_{2,G}^*(f_0,f_0)$, which by Lemma \ref{quasmix} is bounded by $D^{-1/2}$.  As for the first term, we see from \eqref{mus-3} and the pointwise bound $|\tilde \mu_{b,h}(x)| \leq \mu_{b,h}(x)$ that
$$ \|\tilde \mu_{b,h}\|_{S(G)} \leq C_0 D^{-1/2} + \E_{b,h \in G} \sum_{y \in G: \mu_{b,h}(y) \geq C_0/|G|} \mu_{b,h}(y)$$
for each $b,h$. The claim follows.
\end{proof}

\section{The case of $\operatorname{SL}_2$}\label{sl2-sec}

We can now establish the $d=2$ case of Theorem \ref{first}, which serves as a simplified model for the general $d$ case.  From Proposition \ref{quasi-special} and Proposition \ref{gen}, it will suffice to show that
\begin{equation}\label{lo}
\E_{b,h \in G} \sum_{y \in G: \mu_{b,h}(y) \geq C_0/|G|} \mu_{b,h}(y) \ll |F|^{-1}
\end{equation}
for some absolute constant $C_0 \geq 1$, where $\mu_{b,h}$ was defined in \eqref{mobuh}.  

We now need to understand the distribution of $\mu_{b,h}$.  Call an element $b$ of $\operatorname{SL}_2(F)$ \emph{regular semisimple} if its two eigenvalues (in the algebraic closure $\overline{F}$) are distinct, or equivalently if $\tr b \neq \pm 2$.  It is easy to see that all but $O(|F|^2)$ elements of $G$ are regular semisimple.  Since $G$ has cardinality comparable to $|F|^3$, and each of the $\mu_{b,h}$ is normalised in $\ell^1$, we thus see that the contribution of the non-regular semisimple $b$ to \eqref{lo} is $O(|F|^{-1})$, which is acceptable. Thus we may restrict attention to the regular semisimple $b$.  

Now we study the quantity $\mu_{b,h}(y)$.  It is a classical fact that $|F| \ll |Z(b)| \ll |F|$ (this also follows from the Lang-Weil bound, Proposition \ref{lang-weil}).  As such, we have
$$ \mu_{b,h}(y) \ll |F|^{-4} |\{ (g,c) \in G \times Z(b): gc^{-1} h^{-1} g^{-1} c^{-1} h^{-1} = y\}|$$
which we rewrite as
$$ \mu_{b,h}(y) \ll |F|^{-4} |\{ (g,c) \in G \times Z(b): gc^{-1} h^{-1} g^{-1} = y h c\}|$$

If $c^{-1}h^{-1}$ is central (i.e. equal to $\pm 1$), then $y=1$, and the contribution to $\mu_{b,h}(1)$ of this case is $O(|F|^{-1})$.  Now we consider the contribution of those $c$ for which $c^{-1}h^{-1}$ is not central.  Then the centraliser of $c^{-1}h^{-1}$ has cardinality $\gg |F|$, and so every element $k$ of $\operatorname{SL}_2(F)$ of the same trace as $c^{-1}h^{-1}$ has $O(|F|)$ representations of the form $gc^{-1}h^{-1}g^{-1}$.  Of course, if $k$ does not have the same trace as $c^{-1}h^{-1}$, it has no such representations.  We conclude that
$$ \mu_{b,h}(y) \ll |F|^{-1} \delta_{y=1} + |F|^{-3} |\{ c \in Z(b): \tr(yhc) = \tr(c^{-1} h^{-1}) \}|.$$
For $a \in \operatorname{SL}_2(F)$, we see from direct computation (or the Cayley-Hamilton theorem) that $\tr(a^{-1}) = \tr(a)$. We thus have $\mu_{b,h}(y) \ll |F|^{-1}$ for $y=1$, and for $y \neq 1$
we have
$$ \mu_{b,h}(y) \ll |F|^{-3} |\{ c \in Z(b): \tr(yhc) = \tr(hc) \}|.$$
The centraliser $Z(b)$ are the $F$-points of the algebraic variety $\overline{Z(b)} := \{ c \in \operatorname{SL}_2(\overline{F}): cb=bc\}$, which is a curve of complexity\footnote{The complexity of an algebraic variety is defined in Definition \ref{vardef}.} $O(1)$.  From Bezout's theorem, we conclude that the quantity $|\{ c \in Z(b): \tr(yhc) = \tr(hc) \}|$ is bounded by $O(1)$ unless the equation $\tr(yhc) = \tr(hc)$ holds for all $c \in \overline{Z(b)}$, in which case this quantity is bounded instead by $|F|$.  For $C_0$ a sufficiently large absolute constant, we thus have
$$
\sum_{y \in G: \mu_{b,h}(y) \geq C_0/|G|} \mu_{b,h}(y) \ll |F|^{-1} + |F|^{-2} |Y_{b,h}|$$
where $Y_{b,h}$ is the set of all $y \in G$ such that $\tr(yhc) = \tr(hc)$ for all $c \in \overline{Z(b)}$.  It will thus suffice to show that
$$ |Y_{b,h}| \ll |F|$$
whenever $b$ is regular semisimple.

Fix such a $b$.  We may find a basis of $\overline{F}^2$ over $\overline{F}$ that makes $b$ diagonal.  As $b$ is also regular semisimple, we conclude that
$$ \overline{Z(b)} = \left\{ \begin{pmatrix} t & 0 \\ 0 & t^{-1} \end{pmatrix} : t \in \overline{F} \backslash 0 \right\}$$
in this basis, and so the constraint $\tr(yhc)=\tr(hc)$ for all $c \in \overline{Z(b)}$ is equivalent to the requirement that $yh-h$ vanishes on the diagonal.  This constrains $Y_{b,h}$ to a two-dimensional subspace of the four-dimensional vector space $\operatorname{Mat}_{2 \times 2}(\overline{F})$ of $2 \times 2$ matrices; as $y$ also needs to have determinant $1$, we conclude that $Y_{b,h}$ is constrained to a complexity $O(1)$ curve in this plane.  By the Schwarz-Zippel lemma (see Lemma \ref{schwarz}), we conclude that $|Y_{b,h}| \ll |F|$, as required.

\section{The case of $\operatorname{SL}_d$}\label{sld-sec}

Now we turn to the general case of Theorem \ref{first}.  This will basically be a reprise of the arguments in the preceding section, but with a heavier reliance on algebraic geometry in place of \emph{ad hoc} computations.

We allow all implied constants to depend on $d$.  As before, by Proposition \ref{quasi-special} and Proposition \ref{gen}, it suffices to establish the bound \eqref{lo}.  We may assume that $|F|$ is sufficiently large depending on $d$, as the claim is trivial otherwise.

Again, call $b \in \operatorname{SL}_d(F)$ \emph{regular semisimple} if it is diagonalisable in $\overline{F}$ with distinct eigenvalues.  A well-known computation gives
$$ |GL_d(F)| = \prod_{i=0}^{d-1} (|F|^d-|F|^i) = (1+O(|F|^{-1})) |F|^{d^2};$$
since $|G| = |GL_d(F)|/|F^\times|$, we conclude in particular that
\begin{equation}\label{lao}
|F|^{d^2-1} \ll |G| \ll |F|^{d^2-1}
\end{equation}
(this also follows from the Lang-Weil estimate, Proposition \ref{lang-weil}).  If $b$ is not regular semisimple, then its characteristic polynomial has a repeated root.  This constrains $b$ to an algebraic hypersurface of $\operatorname{SL}_d(F)$ of complexity $O(1)$.  This hypersurface has dimension $d^2-2$, so by the Schwarz-Zippel lemma (see Lemma \ref{schwarz}), we have that at most $O(|F|^{d^2-2})$ elements of $G$ are not regular semisimple.  This is only $O(|F|^{-1})$ of the elements of $G$, so to prove \eqref{lo} it suffices as before to consider the contribution of the regular semisimple $b$.

If $b$ is regular semisimple, then the centraliser $Z(b)$ of $b$ consists of the $F$-points of a $d-1$-dimensional torus $\overline{Z(b)}$ in $\operatorname{SL}_d(\overline{F})$, of complexity $O(1)$, defined over $F$.  By the Lang-Weil bound (Proposition \ref{lang-weil}), we have $|F|^{d-1} \ll |Z(b)| \ll |F|^{d-1}$.  Arguing as in the previous section, we thus have
\begin{equation}\label{monh}
 \mu_{b,h}(y) \ll |F|^{-d^2-d+2} |\{ (g,c) \in G \times Z(b): gc^{-1} h^{-1} g^{-1} c^{-1} h^{-1} = y\}|
 \end{equation}
Let $\phi_{b,h}: \operatorname{SL}_d(\overline{F}) \times \overline{Z(b)} \to \operatorname{SL}_d(\overline{F})$ be the map
\begin{equation}\label{phix}
 \phi_{b,h}( g, c) := gc^{-1} h^{-1} g^{-1} c^{-1} h^{-1}.
\end{equation}
This is a regular map of complexity $O(1)$ from the $d^2+d-2$-dimensional irreducible variety $\operatorname{SL}_d(\overline{F}) \times \overline{Z(b)}$ to the $d^2-1$-dimensional variety $\operatorname{SL}_d(\overline{F})$.  

Suppose that $(b,h)$ is such that the map $\phi_{b,h}$ is dominant.  Applying Proposition \ref{quant-dom}, we see that there exists a subset $\Sigma$ of
$\operatorname{SL}_d(\overline{F}) \times \overline{Z(b)}$ which can be covered by $O(1)$ varieties of complexity $O(1)$ and dimension at most $d^2+d-3$, such that for each $y \in \operatorname{SL}_d(\overline{F})$, the set
$$ |\{ (g,c) \in (\operatorname{SL}_d(\overline{F}) \times \overline{Z(b)}) \backslash \Sigma: \phi_{b,h}(g,c) = y \}$$
is covered by $O(1)$ varieties of complexity $O(1)$ and dimension at most $d-1$.  Applying the Schwarz-Zippel bound (Lemma \ref{schwarz}), we conclude that
$$ |\{ (g,c) \in (\operatorname{SL}_d(F) \times Z(b)) \backslash \Sigma: \phi_{b,h}(g,c) = y \}| \ll |F|^{d-1}$$
for all $y \in G$, and thus by \eqref{monh} one has
$$
 \mu_{b,h}(y) \ll |F|^{-d^2+1} + |F|^{-d^2-d+2} |\{ (g,c) \in (G \times Z(b)) \cap \Sigma: gc^{-1} h^{-1} g^{-1} c^{-1} h^{-1} = y\}|.$$
By \eqref{lao}, we conclude (for $C_0$ large enough) that
\begin{align*}
\sum_{y \in G: \mu_{b,h}(y) \geq C_0/|G|} \mu_{b,h}(y) &\ll
|F|^{-d^2-d+2} \sum_{y \in G} |\{ (g,c) \in (G \times Z(b)) \cap \Sigma: gc^{-1} h^{-1} g^{-1} c^{-1} h^{-1} = y\}| \\
&= |F|^{-d^2-d+2} |(G \times Z(b)) \cap \Sigma|
\end{align*}
and hence by another application of Schwarz-Zippel, we have
$$ 
\sum_{y \in G: \mu_{b,h}(y) \geq C_0/|G|} \mu_{b,h}(y)  \ll |F|^{-1}$$
when $\phi_{b,h}$ is dominant.  On the other hand, when $\phi_{b,h}$ is not dominant, we may crudely bound
$$ 
\sum_{y \in G: \mu_{b,h}(y) \geq C_0/|G|} \mu_{b,h}(y) \leq \sum_{y \in G} \mu_{b,h}(y) = 1.$$
To establish \eqref{lo}, it thus suffices to show that there are at most $O( |F|^{-1} |G|^2 )$ pairs $(b,h) \in G \times G$ with $b$ regular semisimple and $\phi_{b,h}$ not dominant.

Fix $b$ regular semisimple.  It suffices to show that $\phi_{b,h}$ is dominant for all but at most $O(|F|^{-1} |G|)$ values of $h \in G$; by the Schwarz-Zippel bound (Lemma \ref{schwarz}), it suffices to show that $\phi_{b,h}$ is dominant for all $h \in \operatorname{SL}_d(\overline{F})$ outside of $O(1)$ algebraic varieties of positive codimension and complexity $O(1)$.  As this assertion only involves $\overline{F}$ and not $F$, we may now diagonalise $b$ over $\overline{F}$, and work in a basis in which $b$ is diagonal (with coefficients in $\overline{F}$ rather than in $F$).  The torus $\overline{Z(b)}$ is now the group $T(\overline{F})$ of diagonal matrices in $\operatorname{SL}_d(\overline{F})$.  It now suffices to establish the following claim:

\begin{proposition}[Quantitative generic non-degeneracy]\label{quant}  Let $k$ be an algebraically closed field, and let $d \geq 1$; we allow all implied constants to depend on $d$.  Then for all $h \in \operatorname{SL}_d(k)$ outside of $O(1)$ algebraic varieties of positive codimension and complexity $O(1)$, the map $\tilde \phi_h: \operatorname{SL}_d(k) \times T(k) \to \operatorname{SL}_d(k)$ defined by
\begin{equation}\label{tphi}
\tilde \phi_{h}( g, c) := gc^{-1} h^{-1} g^{-1} c^{-1} h
\end{equation}
is dominant, where $T(k)$ denotes the group of diagonal matrices in $\operatorname{SL}_d(k)$.
\end{proposition}

Indeed, by setting $k$ equal to the algebraic closure $\overline{F}$ of $F$, and noting that $\phi_{b,h} =\tilde \phi_h h^{-2}$, the claim follows.  (We have shifted $\tilde \phi_h$ in order to map the identity $(1,1)$ to the identity $1$.)

It turns out that by using an ultraproduct argument, one can show that Proposition \ref{quant} is implied by the following, seemingly weaker, qualitative variant of that proposition, in which the uniform bounds on the exceptional set are dropped:

\begin{proposition}[Qualitative generic non-degeneracy]\label{quali}  Let $k$ be an algebraically closed field, and let $d \geq 1$.  Then for generic $h \in \operatorname{SL}_d(k)$ (that is, for all $h$ outside of a finite union of varieties of positive codimension), the map $\tilde \phi_h: \operatorname{SL}_d(k) \times T(k) \to \operatorname{SL}_d(k)$ defined by \eqref{tphi} is dominant.
\end{proposition}

Indeed, if Proposition \ref{quant} failed, then one could find $d \geq 1$ and a sequence $k_n$ of algebraically closed fields such that the set of $h \in \operatorname{SL}_d(k_n)$ for which $\tilde \phi_h$ fails to be dominant cannot be covered by $n$ algebraic varieties of positive codimension and complexity at most $n$.  Performing an ultraproduct with respect to a non-principal ultrafilter on the natural numbers (see \cite[Appendix A]{bgt-product}), we then obtain a new (and much larger) algebraically closed field $k$, with the property that the set of $h \in \operatorname{SL}_d(k)$ for which $\tilde \phi_h$ fails to be dominant cannot be covered by any finite number of algebraic varieties of positive codimension, contradicting Proposition \ref{quali}.  (Here we use the continuity of irreducibility and dominance with respect to ultraproducts; see \cite[Lemma A.2]{bgt-product} and \cite[Lemma A.7]{bgt-product}.)

It remains to prove Proposition \ref{quali}.  By the irreducibility of $\operatorname{SL}_d(\overline{F})$, it suffices to show that the derivative map
$$ D\tilde \phi_h(1,1): \mathfrak{sl}_d(k) \times \mathfrak{t}(k) \to \mathfrak{sl}_d(k)$$
is full rank for generic $h \in \operatorname{SL}_d(k)$, where $\mathfrak{sl}_d(k)$ is the vector space of trace zero $d \times d$ matrices over $k$, and $\mathfrak{t}(k)$ is the subspace of $\mathfrak{sl}_d(k)$ consisting of diagonal matrices over $k$ of trace zero.  From the product rule and \eqref{tphi}, we may evaluate $D\tilde \phi_h(1,1)$ explicitly as
$$ D\tilde\phi_h(1,1)(X,Y) = X - h^{-1} X h - Y - h^{-1} Y h$$
for $X \in \mathfrak{sl}_d(k)$ and $Y \in {\mathfrak t}(k)$.

We may restrict attention to $h$ which are regular semisimple (or equivalently, those $h$ whose characteristic polynomial has no repeated roots), as the complement of this set is certainly contained in a finite number of algebraic varieties of positive codimension.  We may thus diagonalise $h = ADA^{-1}$ for some $A \in \operatorname{SL}_d(k)$ and diagonal $D$ with distinct diagonal entries.  Then we have
$$ D\tilde \phi_h(1,1)(X,Y) = A (  X' - D^{-1} X' D - Y' - D^{-1} Y' D )  A^{-1}$$
where $X' := A^{-1} X A$ and $Y' := A^{-1} Y A$.  We thus see that $D\tilde \phi_h(1,1)$ is full rank if and only if the map
$$ (X',Y') \mapsto X' - D^{-1} X' D - Y' - D^{-1} Y' D$$
is a full rank map from $\mathfrak{sl}_d(\overline{F}) \times A^{-1} {\mathfrak t}(\overline{F}) A$ to $\mathfrak{sl}_d(\overline{F})$.  It thus suffices to show that this map is full rank for generic $A \in \operatorname{SL}_d(k)$ and $D \in T(k)$.

As $D$ is a diagonal matrix with distinct diagonal entries, we see that the image of $\mathfrak{sl}_d(k)$ under the map $X' \mapsto X' - D^{-1} X' D$ is the space of all matrices that vanish on the diagonal.  To show that $D\tilde \phi_h(1,1)$ has full rank, it thus suffices to show that the map $Y' \mapsto \operatorname{diag}(Y' + D^{-1} Y' D)$ has full rank from $A^{-1} {\mathfrak t}(\overline{F}) A$ to ${\mathfrak t}(\overline{F})$.  Since $\operatorname{diag}(Y' + D^{-1} Y' D) = 2 \operatorname{diag}(Y')$, it suffices to show that the diagonal map $Y' \mapsto \operatorname{diag}(Y')$ has full rank from $A^{-1} {\mathfrak t}(\overline{F}) A$ to ${\mathfrak t}(\overline{F})$ for generic $A \in \operatorname{SL}_d(k)$.  As this is clearly a Zariski-open algebraic constraint, and contains the case $A=1$, we conclude that one has full rank for generic $A$, and the claim follows.

\section{Expansion}

In the remarkable paper of Bourgain and Gamburd \cite{bourgain-gamburd}, the quasirandomness properties of $\operatorname{SL}_2(F)$, combined with the product theory in such groups (see \cite{helfgott-sl2}), were used to establish spectral gaps for the generators of various Cayley graphs.  In our notation, the results of \cite{bourgain-gamburd} established spectral gap results, a typical one of which is the assertion that with probability $1-o_{p \to \infty}(1)$, one has
$$ \| \frac{1}{4}(\delta_a + \delta_b + \delta_{a^{-1}} + \delta_{b^{-1}}) \|_{S(\operatorname{SL}_2(F_p))} \leq 1-c$$
for some absolute constant $c>0$, where $F_p$ is a finite field of prime order and $a,b$ is chosen uniformly at random from $\operatorname{SL}_2(F_p)$.  This result has since been generalised in a number of different directions; see \cite{lubo} for a survey. 

In this section, we establish some related expansion results, but instead of a probability measure (such as $\frac{1}{4}(\delta_a + \delta_b + \delta_{a^{-1}} + \delta_{b^{-1}})$) supported on a small number of points, we will establish spectral bounds on (quasi-)probability measures distributed more or less uniformly on subvarieties $V$ of $\operatorname{SL}_d$; this will play an important role in the proof of Theorem \ref{second} in later sections.  The main result is that as long as $V$ is not ``trapped'' in an algebraic subgroup of $\operatorname{SL}_d$ (or a coset thereof), there is a spectral norm bound which gains a power of $|F|$ over the trivial bound.  The arguments are very much in the spirit of Bourgain and Gamburd \cite{bourgain-gamburd}, with the main ingredients being ``escape from subvarieties'', quasirandomness, and some basic algebraic geometry.  However, due to the algebraic structure of the measures being studied, combinatorial tools such as the product theorem of Helfgott \cite{helfgott-sl2} are not required in this argument (though they could certainly be deployed in order to prove more general results, in which the measure in question is not assumed to be adapted to an algebraic subvariety).

More precisely, we will establish the following result.

\begin{proposition}[Expansion from subvarieties]\label{exp}  Let $k$ be an algebraically closed field, and let $F$ be a finite subfield of $k$.  Let $V \subset \operatorname{SL}_d(k)$ be an irreducible algebraic variety defined over $k$ of complexity at most $M$.  Suppose that $V$ is not contained in any coset $Hg$ of a proper algebraic subgroup $H$ of $\operatorname{SL}_d(k)$.  Then one has
$$ \| \mu \|_{S(\operatorname{SL}_d(F))} \ll_{d,M} |F|^{\dim(V)-c} \| \mu\|_{L^\infty(V \cap \operatorname{SL}_d(F))}$$
for all $\mu: \operatorname{SL}_d(F) \to \C$ supported on $V \cap \operatorname{SL}_d(F)$, where $c>0$ depends only on $d$.
\end{proposition}

Recall that $\| \|_{S(G)}$ is the reduced spectral norm, defined in \eqref{reduced}.

\begin{proof}  We perform a downward induction on $\dim(V)$, which is an integer between $0$ and $\dim(\operatorname{SL}_d) = d^2-1$.  When $\dim(V) = \dim(\operatorname{SL}_d)$, the claim follows from \eqref{mus-2}, \eqref{lao}, and Proposition \ref{quasi-special}. Now suppose that $\dim(V) < \dim(\operatorname{SL}_d)$, and that the claim has already been proven for all larger values of $\dim(V)$.

We normalise $\| \mu\|_{L^\infty(V \cap \operatorname{SL}_d(F))} := |F|^{-\dim(V)}$, and allow all implied constants to depend on $d$ and $M$, so our task is now to show that
$$ \| \mu \|_{S(\operatorname{SL}_d(F))} \ll |F|^{-c}.$$

Recall the \emph{$TT^*$ identity}
$$ \|TT^* \|_{\operatorname{op}} = \|T\|_{\operatorname{op}}^2$$
whenever $T$ is a bounded linear operator between Hilbert spaces.  Applying this to the convolution operator $f \mapsto f * \mu$ on the Hilbert space of mean zero functions on $L^2(G)$, we conclude that
$$ \| \mu * \tilde \mu \|_{S(\operatorname{SL}_d(F))} = \| \mu \|_{S(\operatorname{SL}_d(F))}^2$$
where $\tilde \mu: G \to \C$ is the function $\tilde \mu(g) := \overline{\mu(g^{-1})}$.  It will thus suffice to show that
$$  \| \mu * \tilde \mu \|_{S(\operatorname{SL}_d(F))} \ll |F|^{-c}$$
for some $c>0$ depending only on $m,d$.  (Note that as there are only $O(1)$ different values of $\dim(V)$, we may allow the value of the constant $c$ to change with each step of the induction.)  

We consider the product map $\phi: V \times V \to \operatorname{SL}_d(k)$ given by $\phi(v,w) := vw^{-1}$, and let $W'$ be the Zariski closure of $\phi(V \times V)$.  As $V \times V$ is irreducible, $W'$ is also irreducible.  As $W'$ contains a translate of $V$, we have $\dim(W') \geq \dim(V)$.  We claim that we in fact have strict inequality $\dim(W') > \dim(V)$.  To see this, suppose for contradiction that $\dim(W')=\dim(V)$.  Then for each $w \in V$, $Vw^{-1}$ is contained in the irreducible variety $W'$, and has the same dimension as $W'$, and so $Vw^{-1} = W'$ for all $w \in V$.  This implies that $W' (W')^{-1} = \phi(V \times V) \subset W'$, or in other words that $W'$ forms a group, and is thus a proper algebraic subgroup of $\operatorname{SL}_d(k)$.  But $V$ is contained in a coset of $W$, contradicting the hypothesis on $V$.  Thus we have $\dim(W') > \dim(V)$.

We now apply Proposition \ref{quant-dom}, to conclude that $W'$ has complexity $O(1)$, and that there is a subset $\Sigma$ of $V \times V$ covered by $O(1)$ varieties of complexity $O(1)$ and dimension strictly less than $2\dim(V)$, such that for each $w \in W'$, the set $\{ (v,v') \in V \times V \backslash \Sigma: \phi(v,v')= w \}$ is contained in $O(1)$ varieties of complexity $O(1)$ and dimension at most $2\dim(V) - \dim(W')$.  Applying the Schwarz-Zippel bound (Lemma \ref{schwarz}), we conclude that
\begin{equation}\label{sag}
 |\Sigma \cap (G \times G)| \ll |F|^{2\dim(V)-1}
 \end{equation}
and
\begin{equation}\label{sag-2}
 |\{ (v,v') \in ((V  \times V) \cap (G \times G)) \backslash \Sigma: \phi(v,v')= w \}| \ll |F|^{2\dim(V)-\dim(W')}.
 \end{equation}
Next, we expand
$$ \mu * \tilde \mu(w)  = \sum_{(v,v') \in (V \times V) \cap (G \times G): \phi(v,v') = w} \mu(v) \overline{\mu(v')}$$
and then decompose
$$ \mu * \tilde \mu  = \mu_1 + \mu_2$$
where
$$ \mu_1(w) := \sum_{(v,v') \in \Sigma \cap (G \times G): \phi(v,v') = w} \mu(v) \overline{\mu(v')} $$
and
$$ \mu_2(w) := \sum_{(v,v') \in ((V \times V) \cap (G \times G)) \backslash \Sigma: \phi(v,v') = w} \mu(v) \overline{\mu(v')}.$$
As $\|\mu\|_{L^\infty(V)}  = |F|^{-\dim(V)}$, we see that
\begin{equation}\label{mu1g}
\begin{split}
 \|\mu_1\|_{\ell_1(G)} &\leq \sum_{(v,v') \in \Sigma \cap (G \times G)} |F|^{-\dim(V)} |F|^{-\dim(V)} \\
 &\ll |F|^{-1}
 \end{split}
\end{equation}
thanks to \eqref{sag}.  By \eqref{mus}, we thus have
$$
 \|\mu_1\|_{S(G)} \ll |F|^{-1}.$$
Next, from  \eqref{sag-2} and the normalisation $\|\mu\|_{L^\infty(V)}  = |F|^{-\dim(V)}$ we have
$$ \mu_2(w) \ll |F|^{2\dim(V)-\dim(W')} |F|^{-\dim(V)} |F|^{-\dim(V)} = |F|^{-\dim(W')}$$
for all $w\in G$.  As $\mu_2$ is supported on $W'$, we conclude from induction hypothesis that
$$ \|\mu_2\|_{S(G)} \ll |F|^{-c}$$
for some $c>0$ depending only on $d$, and the claim follows.  (Note that as $W'$ contains a translate of $V$, it cannot itself be contained in a coset of a proper algebraic subgroup of $G$.)
\end{proof}

We remark that the above proof in fact allows one to take $c := 2^{-2^{\dim(V)-d}}$.

We will apply Proposition \ref{exp} in the case of a function $\mu$ supported on a conjugacy class:

\begin{corollary}\label{class}  Let $F$ be a finite field, let $d \geq 2$, and let $a \in \operatorname{SL}_d(F)$ be non-central (i.e. $a$ is not a multiple of the identity). Let $C(a) := \{gag^{-1}: g \in \operatorname{SL}_d(F)\}$ be the conjugacy class of $a$.  Then
$$ \| 1_{C(a)} \|_{S(\operatorname{SL}_d(F))} \ll_d |F|^{-c} |C(a)|$$
for some $c>0$ depending only on $d$.
\end{corollary}

\begin{proof}  We allow all implied constants to depend on $d$.  We apply Proposition \ref{exp} with $k$ equal to the algebraic closure of $F$, and $V$ equal to the closed conjugacy class $\overline{C(a)} := \overline{\{gag^{-1}: g \in \operatorname{SL}_d(k)\}}$.  It is clear that $V$ is an irreducible algebraic variety defined over $k$ of complexity $O(1)$; the irreducibility follows since $\operatorname{SL}_d(k)$ is irreducible and the map $g \mapsto gag^{-1}$ is algebraic.  Proposition \ref{exp} will give the desired claim unless $\overline{C(a)}$ is contained in a coset $Hg$ of a proper algebraic subgroup $H$ of $\operatorname{SL}_d(k)$.  But this implies that $H$ contains $\overline{C(a)} \cdot \overline{C(a)}^{-1}$, which implies that the group $N$ generated by $\overline{C(a)} \cdot \overline{C(a)}^{-1}$ is a proper subgroup of $\operatorname{SL}_d(k)$.  But this group is conjugation-invariant and thus normal.  It is a classical fact (see e.g. \cite{humphreys}) that the algebraic group $\operatorname{SL}_d(k)$ is almost simple, in the sense that the only normal subgroups are finite (in fact, the maximal normal subgroup is the center, or equivalently the quotient $P\operatorname{SL}_d(k)$ is simple).  This implies that $\overline{C(a)}$ is finite.  But this contradicts the hypothesis that $a$ is not central, and the claim follows.
\end{proof}

\begin{remark}   A standard application of Schur's lemma gives the identity
$$ \E_{b \in C(a)} \rho(b) = \frac{1}{\dim(V)} (\tr \rho(a)) I_V$$
for any non-trivial irreducible unitary representation $\rho: \operatorname{SL}_d(F) \to U(V)$, where $I_V$ denotes the identity operator on $V$.  From this and Remark \ref{pwt} we see that Corollary \ref{class} is equivalent to the assertion that $|\tr \rho(a)| \ll_d |F|^{-c} \dim(V)$ for any non-trivial irreducible representation $\rho: \operatorname{SL}_d(F) \to U(V)$ and any non-central $a$.  It is likely that this result could also be established directly (with an optimal value of $c$) from the representation theory of $\operatorname{SL}_d(F)$, but we will not do so here.
\end{remark}

\section{A reduction to a Borel group}\label{Reduction-sec}

We will abbreviate $o_{|F| \to \infty}()$ as $o()$ throughout the rest of this paper.

We now begin the proof of Theorem \ref{second} by making some reductions.  The first is to use the Cauchy-Schwarz inequality to reduce Theorem \ref{second} to a seemingly weaker statement in which the absolute values have been moved outside of the $g$ averaging.  In other words, we will deduce Theorem \ref{second} from the following statement:

\begin{theorem}\label{third} Let $F$ be a finite field, and set $G := \operatorname{SL}_2(F)$. Let $S$ denote the set of all elements of $\operatorname{SL}_2(F)$ that are diagonalisable over $F$.  Then for any functions $f_0,f_1,f_2,f_3: G \to \C$, we have
$$
|\E_{g \in S} \E_{x \in G} \prod_{i=0}^3 f_i(xg^{i-1}) - \prod_{i=0}^3 \E_G f_i| \ll o( \prod_{i=0}^3 \|f_i\|_{L^\infty(G)} ).$$
\end{theorem}

Let us assume Theorem \ref{third} for now and see how it implies Theorem \ref{second}.  If $f_3$ is constant, then the claim follows from Theorem \ref{first}, so we may assume without loss of generality that $f_3$ has mean zero.  We may take the $f_i$ to be real-valued, and also normalise $\|f_i\|_{L^\infty(G)}=1$ for each $i$.  Our task is now to show that
$$
\E_{g \in S} |\E_{x \in G} \prod_{i=0}^3 f_i(xg^{i-1})| = o(1).$$
By Cauchy-Schwarz, it suffices to show that
$$
\E_{g \in S} |\E_{x \in G} \prod_{i=0}^3 f_i(xg^{i-1})|^2 = o(1).$$
which we square as
$$
\E_{g \in S} \E_{x,y \in G} \prod_{i=0}^3 f_i(xg^{i-1}) f_i(yg^{i-1}) = o(1).$$
Substituting $y = hx$, we can rewrite the left-hand side as
$$
\E_{h \in G} \E_{g \in S} \E_{x \in G} \prod_{i=0}^3 f_i(xg^{i-1}) f_i(hxg^{i-1}).$$
Applying Theorem \ref{third}, we have
$$ \E_{g \in S} \E_{x \in G} \prod_{i=0}^3 f_i(xg^{i-1}) f_i(hxg^{i-1}) = \prod_{i=0}^3 \E_{x \in G} f_i(x) f_i(hx) + 
o(1)$$
for each $h \in G$, so it suffices to show that
$$ |\E_{h \in G} \prod_{i=0}^3 \E_{x \in G} f_i(x) f_i(hx)| = o(1).$$
We can bound the left-hand side in magnitude by
$$ \E_{h \in G} |\E_{x \in G} f_3(x) f_3(hx)|$$
and the claim now follows from Lemma \ref{quasmix} (applied to the reversed function $x \mapsto f_3(x^{-1})$).
 
It remains to establish Theorem \ref{third}.   We will deduce it from the following variant theorem on the standard Borel subgroup $B$ of $\operatorname{SL}_d(F)$.

\begin{theorem}\label{fourth}  
Let $F$ be a finite field, and $B$ be the subgroup of matrices in $\operatorname{SL}_2(F)$ which are upper-triangular.  Let $U$ be the normal subgroup of $B$ consisting of matrices which are equal to the identity matrix except possibly at the upper right entry.  Let $f_0,\dots,f_3: B \to \C$.  Then
$$ \Lambda_{4,B}(f_0,\dots,f_3) = \Lambda_{4,B}(f_0 * \mu_U,\dots,f_3 * \mu_U) + o( \|f_0\|_{L^\infty(B)} \dots \|f_3\|_{L^\infty(B)} )$$
where $\mu_U := \frac{1}{|U|} 1_U$.
\end{theorem}

Let us assume Theorem \ref{fourth} for now, and show how it implies Theorem \ref{third}.  We may again assume that $f_3$ has mean zero, and that the $f_i$ are real-valued with $\|f_i\|_{L^\infty(G)}=1$ for each $i$.  Our task is to show that
$$
|\E_{g \in S} \E_{x \in G} \prod_{i=0}^3 f_i(xg^{i-1})| = o(1).$$

The first task is to replace the set $S$ by the set $B$ as follows.  Observe that $B$ is the space of all matrices in $\operatorname{SL}_2(F)$ that fix the span $\operatorname{span}(e_2)$ of the second vector $e_2$ of the standard basis $e_1,e_2$ of $F^2$.  Any conjugate $gBg^{-1}$ of $B$, where $g \in \operatorname{SL}_2(F)$, would fix another line; this new line would be identical to the original line $\operatorname{span}(e_2)$ precisely when $g \in B$, so the total number of such conjugates is
$$ |\\operatorname{SL}_2(F)| / |B| = (1+O(|F|^{-1})) |F|.$$

If $g \in S$ is regular semisimple, then it has two distinct one-dimensional eigenspaces in $F$, and thus preserves $2!=2$ distinct lines.  As such, it lies in $gBg^{-1}$ for $2 |B|$ different values of $B$.  We thus see that the number of regular semisimple elements of $S$ is equal to $\frac{|G|}{2|B|}$ times the number of regular semisimple elements of $B$.  An element of $B$ is regular semisimple if and only if its diagonal entries are distinct, so we see that the proportion of elements of $B$ that are regular semisimple is $1-O(|F|^{-1})$.  We conclude that there are $(\frac{1}{2} + O(|F|^{-1})) |G|$ regular semisimple elements of $S$.  As all but $O(|F|^{-1} |G|)$ elements of $G$ (and hence of $S$) are regular semisimple, we thus see that
$$ \E_{g \in S} f(g) = \E_{g \in G} \E_{h \in gBg^{-1}} f(h) + O(|F|^{-1})$$
for any function $f: G \to \C$ of magnitude $O(1)$.  It will thus suffice to show that
$$
\E_{g \in G} \E_{h \in gBg^{-1}} \E_{x \in G} \prod_{i=0}^3 f_i(xh^{i-1}) = o(1).$$
Fix $g \in G$.  By foliating $G$ into left cosets $agBg^{-1}$ of $gBg^{-1}$, and applying Theorem \ref{fourth} (conjugated by $g$) to each coset, we see that
$$ \E_{h \in gBg^{-1}} \E_{x \in agBg^{-1}} \prod_{i=0}^3 f_i(xh^{i-1}) 
= \E_{h \in gBg^{-1}} \E_{x \in agBg^{-1}} \prod_{i=0}^3 (f_i * \mu_{gUg^{-1}})(xh^{i-1}) + o(1)$$
for each $a$.  It thus suffices to show that
$$
\E_{g \in G} \E_{h \in gBg^{-1}} \E_{x \in G} \prod_{i=0}^3 (f_i * \mu_{gUg^{-1}})(xh^{i-1}) = o(1).$$
Applying the crude bound
$$ \left|\E_{h \in gBg^{-1}} \E_{x \in G} \prod_{i=0}^3 (f_i * \mu_{gUg^{-1}})(xh^{i-1})\right| \leq \E_{x \in G} |f_3 * \mu_{gUg^{-1}}(x)|$$
it suffices to show that
$$
\E_{g \in G} \E_{x \in G} |f_3 * \mu_{gUg^{-1}}(x)| = o(1).$$
By Cauchy-Schwarz, it suffices to show that
$$
\E_{g \in G} \E_{x \in G} |f_3 * \mu_{gUg^{-1}}(x)|^2 = o(1).$$
From the identity
$$ \E_{x \in G} |f_3 * \mu_{gUg^{-1}}(x)|^2 = \E_{x\in G} f_3(x) (f_3 * \mu_{gUg^{-1}})(x)$$
it suffices to show that
$$ |\E_{g \in G} \E_{x\in G} f_3(x) (f_3 * \mu_{gUg^{-1}})(x)| = o(1).$$
By definition of the reduced spectral norm, the left-hand side is bounded by
$$ \| \E_{g \in G} \mu_{gUg^{-1}} \|_S.$$
Observe that
$$ \E_{g \in G} \mu_{gUg^{-1}} = \E_{u \in U} \E_{g \in G} \delta_{gug^{-1}} = \E_{u \in U} \frac{1}{|C(u)|} 1_{C(u)}$$
and so by Minkowski's inequality
$$ \| \E_{g \in G} \mu_{gUg^{-1}} \|_S \leq \E_{u \in U} \frac{1}{|C(u)|} \|1_{C(u)} \|_S.$$
By Corollary \ref{class}, we may bound $\frac{1}{|C(u)|} \|1_{C(u)} \|_S$ by $|F|^{-c}$ for some $c>0$ depending only on $d$, except when $u$ is the identity element, in which case we have the trivial bound of $1$.  As $U$ has cardinality $|F|$, we obtain a net bound of $O( |F|^{-1} + |F|^{-c} )$, and the claim follows.

It remains to establish Theorem \ref{fourth}.  This is the purpose of the remaining sections of the paper.

\section{Progressions in a Borel group}

We now prove Theorem \ref{fourth}.  

By splitting each function $f_i$ into functions constant along cosets of $U$, or having mean zero along cosets of $U$, we see that it suffices to show that
$$ \Lambda_{4,B}(f_0,\dots,f_3) = o( \|f_0\|_{L^\infty(B)} \dots \|f_3\|_{L^\infty(B)} )$$
whenever at least one of $f_0,f_1,f_2,f_3$ has mean zero along cosets of $U$.  By the symmetry
$$ \Lambda_{4,B}(f_0,\dots,f_3) = \Lambda_{4,B}(f_3,\dots,f_0)$$
we may assume that $f_{i_0}$ has mean zero along cosets of $U$ for some $i_0 \in \{2,3\}$.  We may also take $f_0,f_1,f_2,f_3$ to be real-valued with $L^\infty(B)$ norm of $1$, so our task is to show that
$$ \E_{x,g \in B} f_0(x) f_1(xg) f_2(xg^2) f_3(xg^3) = o(1).$$

We will take advantage of the short exact sequence
$$ 0 \to F \to B \to F^\times \to 0$$
between the additive group $F = (F,+)$, the Borel group $B$, and the multiplicative group $F^\times := (F \backslash \{0\}, \cdot)$, given by the inclusion map $\psi: F \to B$ and the projection map $\pi: B \to F^\times$ defined by the formulae
$$ \psi(a) := \begin{pmatrix} 1 & a \\ 0 & 1 \end{pmatrix}$$
and
$$ \pi\left( \begin{pmatrix} t & a \\ 0 & t^{-1} \end{pmatrix} \right) = t^{-1}.$$
For any $a,b \in F$, we can make the change of variables $(x,g) \mapsto (\psi(a)x, \psi(b) g)$ and write
\begin{align*}
 \E_{x,g \in B} f_0(x) f_1(xg) f_2(xg^2) f_3(xg^3) &= 
\E_{x,g \in B} f_0\left(\psi(a) x\right) f_1\left(\psi(a) x\psi(b) g\right) \\
&\quad \quad \times f_2\left(\psi(a) x\psi(b) g  \psi(b) g\right) f_3\left(\psi(a) x\psi(b) g  \psi(b) g \psi(b) g\right).
\end{align*}
By using the identity
$$ x \psi(b) = \psi( \pi(x)^2 b ) x$$
for any $x \in B$ and $b \in F$, we can rewrite the above identity as
\begin{align*}
 \E_{x,g \in B} f_0(x) f_1(xg) f_2(xg^2) f_3(xg^3) &= 
\E_{x,g \in B} f_0\left(\psi(a) x\right) f_1\left(\psi(a+\pi(x)^2 b) xg\right) f_2\left(\psi(a + \pi(x)^2 b + \pi(xg)^2 b) xg^2\right) \\
&\quad\quad f_3\left(\psi(a + \pi(x)^2 b + \pi(xg)^2 b + \pi(xg^2) b) xg^3\right).
\end{align*}
On averaging in $a,b$, we conclude that
\begin{align*}
\E_{x,g \in B} f_0(x) f_1(xg) f_2(xg^2) f_3(xg^3) &= 
 \E_{x,g \in B} \E_{a,b \in F} f_{0,x}(a) f_{1,xg}\left(a+\pi(x)^2b\right) f_{2,xg^2}\left(a+\pi(x)^2b+\pi(xg)^2 b\right) \\
 &\quad\quad f_{3,xg^3}\left(a+\pi(x)^2b+\pi(xg)^2b+\pi(xg^2)^2 b\right)
\end{align*}
where $f_{i,x}: F \to \R$ are the functions
$$ f_{i,x}(a) := f_i(\psi(a) x).$$
By dilating $b$ by $\pi(x)^2$, we may simplify the above expression slightly as
\begin{align*}
 \E_{x,g \in B} \E_{a,b \in F} &f_{0,x}(a) f_{1,xg}(a+b) \\
 &\quad f_{2,xg^2}\left(a+(1+\pi(g)^2)b\right) f_{3,xg^3}\left(a+(1+\pi(g)^2+\pi(g)^4) b\right)
 \end{align*}
As is well known, the inner average has too high of a ``complexity'' to be directly treated by Fourier analysis.  However, following Gowers \cite{gowers-4aps}, we may reduce to a form tractable to Fourier analysis after applying the Cauchy-Schwarz inequality.  Indeed, from that inequality we can bound the preceding expression in magnitude by
\begin{align*}
\biggl(\E_{x,g \in B} \E_{a \in F} |\E_{b \in F} &f_{1,xg}\left(a+b\right) f_{2,xg^2}\left(a+(1+\pi(g)^2)b\right) \\
&\quad f_{3,xg^3}(a+(1+\pi(g)^2+\pi(g)^4) b)|^2\biggr)^{1/2}.
\end{align*}
We may expand this expression as
\begin{align*}
&\biggl(\E_{x,g \in B} \E_{a,b,b' \in F} f_{1,xg}(a+b)  f_{1,xg}(a+b') \\
&\quad f_{2,xg^2}\left(a+(1+\pi(g)^2) b\right) f_{2,xg^2}\left(a+(1+\pi(g)^2) b'\right) \\
&\quad f_{3,xg^3}\left(a+(1+\pi(g)^2+\pi(g)^4) b\right) f_{3,xg^3}\left(a+(1+\pi(g)^2+\pi(g)^4) b'\right)\biggr)^{1/2}.
\end{align*}
Writing $b' = b+h$ and shifting $x$ by $g$, this becomes
\begin{align*}
& \biggl(\E_{x,g \in B} \E_{h \in F} \E_{a,b \in F} \Delta_{h} f_{1,x}(a+b) \Delta_{(1+\pi(g)^2) h} f_{2,xg}\left(a+(1+\pi(g)^2)b\right) \\
& \quad \Delta_{(1+\pi(g)^2+\pi(g)^4) h} f_{3,xg^2}\left(a+(1+\pi(g)^2+\pi(g)^4) b\right)\biggr)^{1/2}
\end{align*}
where $\Delta_h f(a) := f(a) f(a+h)$.

Shifting $a$ by $b$, then dilating $b$ by $\pi(g)^{-2}$, we may simplify this slightly as
\begin{align*}
\biggl(\E_{x,g \in B} \E_{h \in F} \E_{a,b \in F} &\Delta_{h} f_{1,x}(a) \Delta_{(1+\pi(g)^2) h} f_{2,xg}(a+b)\\ 
&\quad  \Delta_{(1+\pi(g)^2+\pi(g)^4)h} f_{3,xg^2}\left(a+(1+\pi(g)^2) b\right)\biggr)^{1/2}
\end{align*}
and so our task is now to show that
\begin{equation}\label{targ-1}
\begin{split}
\E_{x,g \in B} \E_{h \in F} \E_{a,b \in F} &\Delta_{h} f_{1,x}(a) \Delta_{(1+\pi(g)^2) h} f_{2,xg}(a+b) \\ 
&\quad \Delta_{(1+\pi(g)^2+\pi(g)^4)h} f_{3,xg^2}\left(a+(1+\pi(g)^2) b\right) = o(1).
\end{split}
\end{equation}

The next step is Fourier expansion.  Consider the trilinear form
$$ \E_{a,b \in F} H_1(a) H_2(a+b) H_3(a+(1+\pi(g)^2) b)$$
for some functions $H_1,H_2,H_3: F \to \C$.  Using some arbitrary non-degenerate bilinear form $\cdot: F \times F \to \R/\Z$, we can form the Fourier series
$$ H_i(a) = \sum_{\xi \in F} \hat H_i(\xi) e(\xi \cdot a)$$
for $i=1,2,3$, where $e(x) := e^{2\pi i x}$ and
$$ \hat H_i(\xi) = \E_{a \in F} H_i(a) e(-\xi \cdot a).$$
Inserting these Fourier series and simplifying, we arrive at the identity
$$ \E_{a,b \in F} H_1(a) H_2(a+b) H_3(a+(1+\pi(g)^2) b) = \sum_{\xi \in F} \hat H_1(\xi) \hat H_2(-(1+\pi(g)^{-2})\xi) \hat H_3(\pi(g)^{-2} \xi).$$
We may thus write the left-hand side of \eqref{targ-1} as
\begin{align*}
\E_{x,g \in B} \E_{h \in F} \sum_{\xi \in F} &(\Delta_{h} f_{1,x})^\wedge(\xi) (\Delta_{(1+\pi(g)^2) h} f_{2,xg})^\wedge \left( -(1+\pi(g)^{-2})\xi \right)\\
& \quad (\Delta_{(1+\pi(g)^2+\pi(g)^4)h} f_{3,xg^2})^\wedge \left( \pi(g)^{-2} \xi \right).
\end{align*}
Splitting off the $\xi=0$ and $\xi\neq 0$ terms, we see that to prove \eqref{targ-1}, it will suffice to establish the bounds
\begin{equation}\label{targ-2}
\E_{x,g \in B} \E_{h \in F} (\Delta_{h} f_{1,x})^\wedge(0) (\Delta_{(1+\pi(g)^2) h} f_{2,xg})^\wedge( 0 )
(\Delta_{(1+\pi(g)^2+\pi(g)^4)h} f_{3,xg^2})^\wedge ( 0 ) = o(1)
\end{equation}
and
\begin{equation}\label{targ-3}
\begin{split}
\E_{x,g \in B} \E_{h \in F} \sum_{\xi \in F^\times} &|(\Delta_{h} f_{1,x})^\wedge(\xi)| |(\Delta_{(1+\pi(g)^2) h} f_{2,xg})^\wedge\left( -(1+\pi(g)^{-2})\xi \right)|\\
&\quad |(\Delta_{(1+\pi(g)^2+\pi(g)^4)h} f_{3,xg^2})^\wedge\left( \pi(g)^{-2} \xi \right)| = o(1).
\end{split}
\end{equation}

\subsection{The contribution of the zero frequency}

We now prove \eqref{targ-2}.  We have
$$ (\Delta_{h} f_{1,x})^\wedge(0) = \E_{a \in F} f_{1,x}(a) f_{1,x}(a+h)$$
and thus by Fourier expansion
$$ (\Delta_{h} f_{1,x})^\wedge(0) = \sum_{\xi_1 \in F} |\hat f_{1,x}(\xi_1)|^2 e(\xi_1 \cdot h).$$
Similarly we have
$$ (\Delta_{(1+\pi(g)^2) h} f_{2,xg})^\wedge(0) = \sum_{\xi_2 \in F} |\hat f_{2,xg}(\xi_2)|^2 e((1+\pi(g)^2)^*) \xi_2 \cdot h)$$
and
$$ (\Delta_{(1+\pi(g)^2+\pi(g)^4) h} f_{2,xg})^\wedge(0) = \sum_{\xi_3 \in F} |\hat f_{3,xg^2}(\xi_3)|^2 e((1+\pi(g)^2+\pi(g)^4) + \rho(xg^2x^{-1})^* \xi_3 \cdot h).$$
 Inserting these identities and performing the $h$ averaging, we conclude that the left-hand side of \eqref{targ-2} can be rewritten as
$$ \E_{x,g \in B} \sum_{\xi_1,\xi_2,\xi_3 \in F: \xi_1 + (1+\pi(g)^2) \xi_2 + (1+\pi(g)^2+\pi(g)^4) \xi_3 = 0}
|\hat f_{1,x}(\xi_1)|^2 |\hat f_{2,xg}(\xi_2)|^2 |\hat f_{3,xg^2}(\xi_3)|^2.$$
Recall that $f_{i_0}$ was assumed to have mean zero on cosets of $H$, which implies that we may restrict $\xi_{i_0}$ to be non-zero. We note that the quantity $|\hat f_{i,x}(\xi)|^2$ is unchanged if one multiplies $x$ on the left (or right) by an element of $U$, and so we may write
$$ |\hat f_{i,x}(\xi)|^2 = \mu_{i, \pi(x)}(\xi)$$
for some non-negative quantity $\mu_{i,t}(\xi)$, defined for $i=1,2,3$, $t \in F^\times$, and $\xi \in F$.  We can then simplify the previous expression as
\begin{equation}\label{stf}
 \E_{s,t \in F^\times} \sum_{\xi_1,\xi_2,\xi_3 \in F: \xi_1 + (1+t^2) \xi_2 + (1+t^2+t^4) \xi_3 = 0; \xi_{i_0} \neq 0}
\mu_{1,s}(\xi_1) \mu_{2,st}(\xi_2) \mu_{3,st^2}(\xi_3).
\end{equation}
To show that this expression is $o(1)$, it will suffice to establish the combinatorial bound
\begin{equation}\label{del}
\E_{s,t \in F^\times} 1_{\eta_1(s) + (1+t^2) \eta_2(st) + (1+t^2+t^4) \eta_3(st^2) = 0} = o(1)
\end{equation}
for any choice of functions $\eta_i: F^\times \to F$ for $i=1,2,3$, with $\eta_{i_0}$ non-zero.  Indeed, by the Plancherel identity we have
$$ \sum_\xi \mu_{i,s}(\xi) \leq 1$$
for all $i=1,2,3$ and $s \in F^\times$, with $\mu_{i_0,s}(0)=0$, so we may find random functions $\eta_i: F^\times \to F$ with $\eta_{i_0}$ nowhere vanishing, and with the property that
$$ \mu_{i,s}(\xi) \leq \P( \eta_i(s) = \xi )$$
for all $i=1,2,3$ and $s \in F^\times$.  Applying \eqref{del} with these functions and taking expectations, we conclude that the quantity \eqref{stf} is $o(1)$ as desired.

It remains to establish \eqref{del}, which is a bound of ``sum-product'' type, in that it is asserting a certain combinatorial incompatibility between the multiplicative and additive structures on $F$.  Assume for contradiction that we can find arbitrarily large finite fields $F$ and functions $\eta_1,\eta_2,\eta_3: F^\times \to F$ with $\eta_{i_0}$ nowhere vanishing, for which
$$ \E_{s,t \in F^\times} 1_{\eta_1(s) + (1+t^2) \eta_2(st) + (1+t^2+t^4) \eta_3(st^2) = 0} \gg 1.$$
Fix $F, \eta_1,\eta_2,\eta_3$.  Let $A \subset (F^\times)^2$ be the set of all pairs $(s,t)$ for which
$$ \eta_1(s) + (1+t^2) \eta_2(st) + (1+t^2+t^4) \eta_3(st^2) = 0,$$
thus $|A| \gg |F^\times|^2$.  Applying the multidimensional Szemer\'edi theorem (Theorem \ref{multi}) to the multiplicative group $F^\times$, we conclude that there are $\gg |F|^3$ triples $(s,t,r)$ with the property that $(sr^i, tr^j) \in A$ for all $-100 \leq i,j\leq 100$ (say), thus
\begin{equation}\label{srj}
 \eta_1(sr^i) + (1+r^{2j} t^2) \eta_2(str^{i+j}) + (1+r^{2j} t^2+ r^{4j} t^4) \eta_3(st^2r^{i+2j}) = 0
\end{equation}
for all $-100 \leq i,j \leq 100$.  We will eliminate the $\eta_i$ terms from \eqref{srj} (taking advantage of the non-vanishing nature of $\eta_{i_0}$) to obtain a non-trivial algebraic constraint on $s,t,r$, which will contradict the assertion that $\gg |F|^3$ triples $(s,t,r)$ exist with this property if $|F|$ is large enough.

We turn to the details.  Fix $s,t,r$ obeying \eqref{srj}.  If we abbreviate $\eta_k(st^{k-1} r^i)$ as $c_k(i)$, and also write $\alpha_j := 1+r^{2j} t^2$ and $\beta_j := 1+r^{2j} t^2 + r^{4j} t^4$, we have
$$ c_1(i) + \alpha_j c_2(i+j) + \beta_j c_3(i+2j) = 0$$
for all $-100 \leq i,j \leq 100$.  In particular, applying this identity for $j$ and $j+1$ and subtracting, we have
$$ \alpha_{j+1} c_2(i+j+1) - \alpha_j c_2(i+j) = \beta_j c_3(i+2j) - \beta_{j+1} c_3(i+2j+2) $$ 
for all $-90 \leq i,j \leq 90$ (say).  Replacing $(i,j)$ by $(i-2,j+2)$, $(i+2,j-1)$, and $(i, j+1)$, we obtain the system of four equations
\begin{align}
\alpha_{j+1} c_2(i+j+1) - \alpha_j c_2(i+j) &= \beta_j c_3(i+2j)  - \beta_{j+1} c_3(i+2j+2) \label{cj-1}\\
\alpha_{j+3} c_2(i+j+1) - \alpha_{j+2} c_2(i+j) &= \beta_{j+2} c_3(i+2j+2)  - \beta_{j+3} c_3(i+2j+4) \label{cj-2}\\
\alpha_j c_2(i+j+2) - \alpha_{j-1} c_2(i+j+1) &= \beta_{j-1} c_3(i+2j)  - \beta_j c_3(i+2j+2)\label{cj-3}\\
\alpha_{j+2} c_2(i+j+2) - \alpha_{j+1} c_2(i+j+1) &= \beta_{j+1} c_3(i+2j+2)  - \beta_{j+2} c_3(i+2j+4) \label{cj-4}
\end{align}
for all $-80 \leq i,j \leq 80$ (say).

We now eliminate the various $c_2$ factors in this system to obtain a linear recurrence in the $c_j$.  Multiplying \eqref{cj-1} by $\alpha_{j+2}$ and \eqref{cj-2} by $\alpha_j$ and subtracting to eliminate the $c_2(i+j)$ term, we conclude that
\begin{equation}\label{doo}
 (\alpha_{j+1}\alpha_{j+2}-\alpha_{j+3}\alpha_j) c_2(i+j+1) = \beta_j \alpha_{j+2} c_3(i+2j) - (\beta_{j+1} \alpha_{j+2}+\beta_{j+2} \alpha_j) c_3(i+2j+2) + \beta_{j+3} \alpha_j c_3(i+2j+4).
\end{equation}
Similarly, if we multiply \eqref{cj-3} by $\alpha_{j+2}$ and \eqref{cj-4} by $\alpha_{j}$ and subtract to eliminate the $c_j(i+j+2)$ term, we have
$$ (\alpha_j \alpha_{j+1} - \alpha_{j-1}\alpha_{j+2}) c_2(i+j+1) = \beta_{j-1} \alpha_{j+2} c_3(i+2j) - (\beta_j \alpha_{j+2} + \beta_{j+1} \alpha_j) c_3(i+2j+2) + \beta_{j+2} \alpha_j c_3(i+2j+4).$$
A brief calculation reveals that
$$ \alpha_{j+1}\alpha_{j+2}-\alpha_{j+3}\alpha_j =r^2 (\alpha_{j}\alpha_{j+1}-\alpha_{j+2}\alpha_{j-1})$$
and so we may also eliminate $c_2(i+j+1)$ and conclude that
\begin{equation}\label{loo}
\beta'_j \alpha_{j+2} c_3(i+2j) - (\beta'_{j+1} \alpha_{j+2}+\beta'_{j+2} \alpha_j) c_3(i+2j+2) + \beta'_{j+3} \alpha_j c_3(i+2j+4) = 0
\end{equation}
for all $-80 \leq i,j \leq 80$, where
$$ \beta'_j := \beta_j - r^2 \beta_{j-1} = (1-r^{-2}) (r^{4j} t^4-r^2).$$
We continue the elimination process.  Applying \eqref{loo} with $(i,j)$ replaced by $(i+2,j-1)$, we conclude that
$$
\beta'_{j-1} \alpha_{j+1} c_3(i+2j) - (\beta'_{j} \alpha_{j+1}+\beta'_{j+1} \alpha_{j-1}) c_3(i+2j+2) + \beta'_{j+2} \alpha_{j-1} c_3(i+2j+4) = 0
$$
for all $-70 \leq i,j \leq 70$ (say).  Multiplying this equation by $\beta'_{j+3} \alpha_j$ and \eqref{loo} by $\beta'_{j+2} \alpha_{j-1}$ and subtracting, we conclude that
\begin{align*}
&(\beta'_{j-1} \beta'_{j+3} \alpha_j \alpha_{j+1} - \beta'_j \beta'_{j+2} \alpha_{j-1} \alpha_{j+2}) c_3(i+2j)\\
&= (\beta'_{j} \beta'_{j+3} \alpha_j \alpha_{j+1}+\beta'_{j+1} \beta'_{j+3} \alpha_{j-1} \alpha_j - \beta'_{j+1} \beta'_{j+2} \alpha_{j-1} \alpha_{j+2} - (\beta'_{j+2})^2 \alpha_{j-1} \alpha_j) c_3(i+2j+2)
\end{align*}
for all $-70 \leq i,j \leq 70$.

We apply this with $(i,j)$ replaced by $(i-2,1)$ and $(i-4,2)$ to conclude that
$$
(\beta'_0 \beta'_4 \alpha_1 \alpha_2 - \beta'_1 \beta'_3 \alpha_0 \alpha_3) c_3(i)
= (\beta'_1 \beta'_4 \alpha_1 \alpha_2+\beta'_2 \beta'_4 \alpha_0 \alpha_1 - \beta'_2 \beta'_3 \alpha_0 \alpha_3 - (\beta'_3)^2 \alpha_0 \alpha_1) c_3(i+2)$$
and
$$
(\beta'_1 \beta'_5 \alpha_2 \alpha_3 - \beta'_2 \beta'_4 \alpha_1 \alpha_4) c_3(i)
= (\beta'_2 \beta'_5 \alpha_2 \alpha_3 + \beta'_3 \beta'_5 \alpha_1 \alpha_2 - \beta'_3 \beta'_4 \alpha_1 \alpha_3 - (\beta'_4)^2 \alpha_1 \alpha_2) c_3(i+2)$$
for all $-60 \leq i \leq 60$ (say).  Eliminating $c_3(i+2)$, we conclude that either $c_3(i)$ vanishes for all $-60 \leq i \leq 60$, or else we have
the constraint
\begin{align*}
& (\beta'_0 \beta'_4 \alpha_1 \alpha_2 - \beta'_1 \beta'_3 \alpha_0 \alpha_3) (\beta'_2 \beta'_5 \alpha_2 \alpha_3 + \beta'_3 \beta'_5 \alpha_1 \alpha_2 - \beta'_3 \beta'_4 \alpha_1 \alpha_3 - (\beta'_4)^2 \alpha_1 \alpha_2) \\
&= (\beta'_1 \beta'_5 \alpha_2 \alpha_3 - \beta'_2 \beta'_4 \alpha_1 \alpha_4) (\beta'_1 \beta'_4 \alpha_1 \alpha_2+\beta'_2 \beta'_4 \alpha_0 \alpha_1 - \beta'_2 \beta'_3 \alpha_0 \alpha_3 - (\beta'_3)^2 \alpha_0 \alpha_1).
\end{align*}
After eliminating some factors of $(1-r^{-2})$, this is a polynomial constraint between $r$ and $t$ of bounded degree.  One can easily verify that the constraint is not a tautology (for instance, setting $r=2$ and $t=2$, the left-hand side is approximately $-1.96 \times 10^{24}$ and the right-hand side is approximately $3.61 \times 10^{32}$).  Thus, by the Schwarz-Zippel lemma, there are only $O(|F|)$ possible pairs $(r,t)$, and thus $O(|F|^2)$ triples $(r,s,t)$, that obey this constraint.  Outside of those exceptional triples, we thus have $c_3(i)$ vanishing for all $-60 \leq i \leq 60$.  Applying \eqref{doo}, we conclude that $c_2(0)$ vanishes as well, unless $\alpha_1 \alpha_2-\alpha_3 \alpha_0$ vanishes.   The latter possibility is also a bounded degree non-tautological constraint on $r,t$ and so also only occurs for $O(|F|^2)$ triples $(r,s,t)$.  Thus we see that $c_3(0)$ and $c_2(0)$ both vanish outside of these exceptional triples.  But this contradicts the assumption that $\eta_{i_0}$ never vanishes (recall that $i_0$ is either $2$ or $3$).   We have thus demonstrated that there are at most $O(|F|^2)$ triples $(r,s,t)$ for which \eqref{srj} holds for all $-100 \leq i,j \leq 100$.  But we also know that there are $\gg |F|^3$ such triples, leading to a contradiction for $|F|$ sufficiently large, as required.

\subsection{The contribution of the non-zero frequencies}

Finally, we prove \eqref{targ-3}.  This will be done by a variant of the Cauchy-Schwarz arguments used to establish Theorem \ref{first}. Observe that one multiplies $x \in G$ on the left by some element $\psi(k)$ of $U$, then $f_{i,x}$ and $\Delta_h f_{i,x}$ become translated by $k$, and the quantity $|\widehat{\Delta_{h} f_{i,x}}(\xi)|$ is unchanged.  Thus, for any $i=1,2,3$, $x \in G$, $h \in F$, and $\xi \in F^\times$, we may write
\begin{equation}\label{hdef}
|\widehat{\Delta_{h} f_{i,x}}(\xi)| = H_{i,h,\pi(x)}(\xi)
\end{equation}
for some function $H_{i,h,\pi(x)}: F^\times \to \R^+$ depending on $h$ and $\pi(x)$.  We may thus rewrite \eqref{targ-3} as
$$
\E_{s \in F^\times} \E_{h \in F} \sum_{\xi \in F^\times} H_{1,h,s}(\xi) 
\E_{t \in F^\times} H_{2,(1+t^4)h, st}( - (1+t^{-4})\xi ) H_{3,(1+t^4+t^8)h, st^2} ( t^{-4} \xi ) = o(1).$$
From Plancherel we have
$$ \sum_{\xi \in F^\times} H_{1,h,s}(\xi)^2 \leq 1$$
for all $s \in F^\times$ and $h \in F$, so by Cauchy-Schwarz it suffices to show that
$$
\E_{s \in F^\times} \E_{h \in F} \sum_{\xi \in F^\times} 
|\E_{t \in F^\times} H_{2,(1+t^4)h, st}( - (1+t^{-4})\xi ) H_{3,(1+t^4+t^8)h, st^2} ( t^{-4} \xi )|^2 = o(1),$$
which we expand as
\begin{align*}
&\E_{s,t,u \in F^\times} \E_{h \in F} \sum_{\xi \in F^\times} 
H_{2,(1+t^4)h, st}( - (1+t^{-4})\xi )^2 H_{3,(1+t^4+t^8)h, st^2} ( t^{-4} \xi )^2 \\
&\quad H^2_{2,(1+u^4)h, su}( - (1+u^{-4})\xi ) H^2_{3,(1+u^4+u^8)h, su^2} ( u^{-4} \xi ) 
= o(1).
\end{align*}
By another Cauchy-Schwarz and symmetry, it thus suffices to show that
$$
\E_{s,t,u \in F^\times} \E_{h \in F} \sum_{\xi \in F^\times} 
H^4_{2,(1+t^4)h, st}( - (1+t^{-4})\xi ) H^4_{3,(1+u^4+u^8)h, su^2} ( u^{-4} \xi ) 
= o(1).$$
There are at most $4$ values of $t$ for which $t^4 = -1$, and each of these values of $t$ contributes $O(|F|^{-1})$ to the above sum (using Plancherel's theorem $\sum_{\xi} H_{i,h,s}(\xi) \leq 1$ and the trivial bound $H_{i,h,s}(\xi) \leq 1$), and may be discarded.
Dilating $h, s, \xi$ by $(1+t^4)^{-1}$, $t^{-1}$, $-(1+t^{-4})$ respectively, we rewrite the remaining component of the above estimate as
$$
\E_{s,t,u \in F^\times} \E_{h \in F} \sum_{\xi \in F^\times} 1_{t^4 \neq -1}
H^4_{2,h,s}(\xi) H^4_{3,(1+u^4+u^8) (1+t^4)^{-1} h, st^{-1}u^2} (  - (1+t^{-4})^{-1} u^{-4} \xi ) 
= o(1).$$
Making the change of variables $(s,u,v) := (s,u,st^{-1}u^2)$, so that $t = s u^2 v^{-1}$, this becomes
$$
\E_{s,u,v \in F^\times} \E_{h \in F} \sum_{\xi \in F^\times} 1_{s^4 u^8 v^{-4} \neq -1}
H^4_{2,h,s}(\xi) H^4_{3,(1+u^4+u^8) (1+s^4 u^8 v^{-4})^{-1} h, v} (  - (1+s^{-4} u^{-8} v^{4})^{-1} u^{-4} \xi ) 
= o(1).$$
From Plancherel's theorem and the trivial bound $H_{2,h,s}(\xi) \leq 1$ we have
$$ \sum_{\xi \in F^\times} H^4_{2,h,s}(\xi)  \leq 1$$
for each $h \in F$ and $s \in F^\times$.  It will thus suffice to establish the bound
$$
\E_{u \in F^\times} 1_{s^4 u^8 v^{-4} \neq -1} H^4_{3,(1+u^4+u^8) (1+s^4 u^8 v^{-4})^{-1} h, v} (  - (1+s^{-4} u^{-8} v^{4})^{-1} u^{-4} \xi ) 
= o(1) $$
for all $\xi \in F^\times$, and all but at most $o(|F|^3)$ choices of $(s,v,h) \in F^\times \times F^\times \times F$.

Fix $s,v,h$.  Our task is to show that for all but $o(|F|^3)$ choices of $(s,v,h)$, one has
\begin{equation}\label{ua-0}
 \E_{u \in F} 1_A(u) H^4_{3,\phi(u), v}( \eta(u) )|^4 = o(1)
\end{equation}
where $A := \{ u \in F^\times: s^4 u^8 v^{-4} \neq -1\}$,
$$ \phi(u) := (1+u^4+u^8) (1+s^4 u^8 v^{-4})^{-1} h,$$
and
$$ \eta(u) := - (1+s^{-4} u^{-8} v^{4})^{-1} u^{-4} \xi.$$
We may assume that $h$ is non-zero, as this only excludes $O(|F|^2) = o(|F|^3)$ values of $(s,v,h)$.

If we write $f := f_{3,g}$ for some $g \in \pi^{-1}(v)$ and expanding out the definition \eqref{hdef} of $H_{3,h,s}$, we may rewrite \eqref{ua-0} as
\begin{equation}\label{ua}
 \E_{u \in F} 1_A(u) |\widehat{\Delta_{\phi(u)} f}(\eta(u))|^4 = o(1).
\end{equation}

The next step is to apply the Cauchy-Schwarz inequality again, in the spirit of the work of Gowers \cite{gowers-4aps}.  First, to show \eqref{ua}, it will suffice to show (using the trivial bound $|\widehat{\Delta_h f}(\eta)| \leq 1$) that
$$
 \E_{u \in F} 1_A(u) |\widehat{\Delta_{\phi(u)} f}( \eta(u) )| = o(1)$$
 or equivalently that
$$
\E_{u \in F} b(u) \widehat{\Delta_{\phi(u)} f}( \eta(u) ) = o(1)$$
 for any function $b: F \to \R$ supported on $A$ with $|b(u)| \leq 1$ for all $u$. We can expand the left-hand side as
$$
\E_{x,u \in F} b(u) f(x) f(x+\phi(u)) e( - \eta(u) \cdot x ),$$
and rearrange this as
$$ \E_{x,y \in F} f(x) f(y) K(x,y)$$
where
$$ K(x,y) := \sum_{u \in F: \phi(u)=y-x} b(u) e( -\eta(u) \cdot x).$$
Applying the Cauchy-Schwarz inequality twice, and using the boundedness of $f$, we have
$$ |\E_{x,y \in F} f(x) f(y) K(x,y)|^4 \leq \E_{x,y,x',y' \in F} K(x,y) \overline{K(x,y')} \overline{K(x',y)} K(x',y')$$
so it will suffice to show that
$$ \E_{x,y,x',y' \in F} \overline{K(x,y')} \overline{K(x',y)} K(x',y') = o(1).$$
The left-hand side may be expanded as
\begin{align*}
&|F|^{-4} \sum_{u_1,u_2,u_3,u_4 \in A} 
b(u_1) b(u_2) b(u_3) b(u_4)\\
&\quad \sum_{x,y,x',y' \in F: \phi(u_1) = x-y, \phi(u_2) = x-y', \phi(u_3) = x'-y, \phi(u_4) = x'-y'}\\
&\quad\quad e( - (\eta(u_1) - \eta(u_2) - \eta(u_3) + \eta(u_4)) \cdot x) e( (\eta(u_3) - \eta(u_4)) \cdot (x'-x) ).
\end{align*}
The quantity $x'-x$ in the summand is equal to $\phi(u_3)-\phi(u_1)$, and so this phase is constant over the inner summation.  By Fourier analysis, we see that the inner summation is thus $O(|F|)$ when $\eta(u_1) + \eta(u_4) = \eta(u_2) + \eta(u_3)$ and $\phi(u_1) + \phi(u_4) = \phi(u_2) + \phi(u_3)$, and zero otherwise.  It thus suffices to show that
$$ | \{ (u_1,u_2,u_3,u_4) \in A^4: \eta(u_1) + \eta(u_4) = \eta(u_2) + \eta(u_3); \phi(u_1) + \phi(u_4) = \phi(u_2) + \phi(u_3) \}| = o( |F|^3 ).$$
Canceling out the non-zero $h$ and $\xi$ factors, and replacing each of the $u_i$ by their fourth powers (at the cost of paying $O(1)$ in the cardinality bound), this becomes
$$ | \{ (u_1,u_2,u_3,u_4) \in A^4: \Phi(u_1) + \Phi(u_4) = \Phi(u_2) + \Phi(u_3) \}| = o(|F|^3 )$$
where $\Phi: F \to F^2$ is the rational function
$$ \Phi(u) := ( (1+u+u^2) (1 + k u^2)^{-1}, (1+k^{-1}u^{-2})^{-1} u^{-1})$$
and $k := s^4 v^{-4}$.  We can simplify $(1+k^{-1}u^{-2})^{-1} u^{-1}$ as $ku (1+ku^2)^{-1}$ and $(1+u+u^2) (1+ku^2)^{-1}$ as $k^{-1} + (1-k^{-1} + u) (1+ku^2)^{-1}$, so after excluding the $O(|F|^2)=o(|F|^3)$ triplets $(s,v,h)$ for which $k=1$, we may replace $\Phi$ by
$$ \tilde \Phi(u) := ( (1+ku^2)^{-1}, u (1+ku^2)^{-1} ).$$
This function takes values in the conic section
$$ C := \{ (x,y) \in F: x^2 + k y^2 = x \}$$
with each point in $C$ arising from at most two values of $u$, and so it suffices to show that
$$ |\{ (p_1,p_2,p_3,p_4) \in C^4: p_1+p_4 = p_2 + p_3 \}| = o( |F|^3 ).$$
But from Bezout's theorem we see that each point in $F^2$ can be expressed in at most two ways as the sum of two elements in $C$, and so the left-hand side is $O(|F|^2)$, and the claim follows.

\begin{remark} The above argument in fact allows us to replace $o(1)$ by $O(|F|^{-c})$ for some absolute constant $c>0$, for the contribution of the non-zero frequencies $\xi$.  Unfortunately, due to the reliance on the multidimensional Szemer\'edi theorem, we are unable to obtain a similarly strong bound for the contribution of the zero frequencies.
\end{remark}

\appendix

\section{Some algebraic geometry}

Throughout this appendix, $k$ is an algebraically closed field, and $F$ is a finite subfield of $k$.  The purpose of this appendix is to review some basic algebraic geometry regarding varieties and regular maps over $k$.

We begin with the definition of a variety.  For the purposes of this paper, we may restrict attention to affine varieties for simplicity, but most of the results here can be extended to other types of varieties (projective, quasiprojective, etc.).

\begin{definition}[Varieties]\label{vardef}  An \emph{(affine) variety} defined over $k$ is a subset $V \subseteq k^n$ of the form
$$ V = \{ x \in k^n: P_1(x) = \dots = P_m(x) = 0 \}$$
where $n,m$ are natural numbers, and $P_1,\dots,P_m: k^n \to k$ are polynomials.  We say that the variety has \emph{complexity at most $M$} if $n,m$ are at most $M$, and all the degrees of $P_1,\dots,P_m$ are at most $M$.  If furthermore the polynomials $P_1,\dots,P_m$ have coefficients defined over $F$, we say that $V$ is \emph{defined over $F$} (with complexity at most $M$).   A variety is \emph{(geometrically) irreducible} if it cannot be expressed as the union of two strictly smaller subvarieties.

The \emph{Zariski closure} of a subset $E$ of $k^n$ is defined to be the intersection of all the varieties in $k^n$ that contain $E$.

The \emph{dimension} of a non-empty variety $V \subset k^n$ is the largest natural number $d$ for which one has a chain
$$ \emptyset \subsetneq V_0 \subsetneq \dots \subsetneq V_d \subset V$$
of irreducible varieties $V_0,\dots,V_d$.  We adopt the convention that the empty set has dimension $-\infty$.
\end{definition}

We have the following basic upper bound for the number of $F$-points on a variety:

\begin{proposition}[Schwarz-Zippel bound]\label{schwarz}  Let $V \subset k^m$ be an affine variety defined over $k$ of complexity at most $M$ and dimension $d$.  Then
$$ |V \cap F^m| \ll_{m,M} |F|^d.$$
\end{proposition}

\begin{proof}
See for instance \cite[Lemma 1]{lang}.  One can make the implied constant depend linearly on the degree of $V$, but we will not need this refinement here.
\end{proof}

In the case that $V$ is irreducible and defined over $F$, we have the following well-known refinement of Proposition \ref{schwarz}:

\begin{proposition}[Lang-Weil bound]\label{lang-weil}  Let $V \subset k^m$ be a geometrically irreducible affine variety defined over $F$ of complexity at most $M$ and dimension $d$.  Then
$$ |V \cap F^m| = (1 + O_{m,M}(|F|^{-1/2})) |F|^d.$$
In particular, if $|F|$ is sufficiently large depending on $m,M$, one has
$$ |F|^d \ll |V \cap F^m| \ll |F|^d.$$
\end{proposition}

\begin{proof}
See \cite[Theorem 1]{lang}.  Again, more precise versions of the error term are available, but we will not need them here.
\end{proof}

Now we recall the notions of regular and dominant maps between varieties.  Our definition will be somewhat complicated due to the need to assign quantitative complexities to such maps.

\begin{definition}[Regular map]\label{regmap}  Let $V \subset k^n$ and $W \subset k^m$ be affine varieties, and let $M \geq 1$.  A map $f: V \to W$ is said to be \emph{regular} with complexity at most $M$ if $V, W$ are individually of complexity at most $M$, and if one can cover $V$ by some varieties $V_1,\dots,V_r$ of complexity at most $M$ for some $r \leq M$ such that for each $1 \leq j \leq r$, the map $f|_{V_j}$ has the form $(P_{j,1}/Q_{j,1},\dots, P_{j,m}/Q_{j,m})$, where the $P_{j,l}, Q_{j,l}$ are homogeneous polynomial maps from $k^{n+1}$ to $k$ with $\deg(P_{j,l}) = \deg(Q_{j,l}) \leq M$, and the $Q_{j,l}$ are non-vanishing on $V_j$.

A regular map $\phi: V \to W$ is \emph{dominant} if $V$ is irreducible and $\phi(V)$ is Zariski-dense in $W$.
\end{definition}

The following proposition asserts (in a certain technical quantitative sense) that regular maps are always ``essentially dominant'' after a reduction in the range, and that the fibres of such maps usually have the expected dimension.

\begin{proposition}[Quantitative dominance]\label{quant-dom}  Let $V \subset k^m, W \subset k^n$ be algebraic varieties defined over $k$ of complexity at most $M$, with $V$ irreducible and let $\phi: V \to W$ be a regular map of complexity at most $M$.  Then there exists a subset $V'$ of $V$ and an irreducible subvariety $W'$ of $W$ of complexity $O_{M}(1)$, with the following properties:
\begin{itemize}
\item[(i)] (Zariski density) $V \backslash V'$ can be covered by the union of $O_{M}(1)$ varieties of complexity $O_{M}(1)$ and dimension strictly less than $\dim(V)$.
\item[(ii)]  (Controlled image)  $W'$ is equal to the Zariski closure of $\phi(V)$; in particular, $\phi: V\to W'$ is a dominant map.
\item[(iii)]  (Controlled fibres)  For each $w \in W'$, the set $\{ v \in V': \phi(v) = w \}$ can be covered by the union of $O_{M}(1)$ varieties of complexity $O_{M}(1)$ and dimension at most $\dim(V)-\dim(W')$.
\end{itemize}
\end{proposition}

\begin{proof}  This follows from \cite[Lemma 3.7]{bgt-product}.
\end{proof}

\section{A quantitative multidimensional Szemer\'edi theorem}

The purpose of this section is to establish the following multidimensional Szemer\'edi theorem:

\begin{theorem}[Multidimensional Szemer\'edi theorem]\label{multi}  Let $G = (G,+)$ be an additive group, let $k,m \geq 1$ be integers, and let $A \subset G^m$ be a set with $|A| \geq \delta |G|^m$.  Then there are $\gg_{k,m,\delta} |G|^{m+1}$ tuples $(a_1,\dots,a_m,r) \in G^{m+1}$ with the property that
$$ (a_1+i_1 r, \dots, a_m+i_m r) \in A$$
for all integers $i_1,\dots,i_m \in \{-k,\dots,k\}$.
\end{theorem}

This is a variant of the multidimensional Szemer\'edi theorem of Furstenberg and Katznelson \cite{fk0}.  There are now many techniques to establish such results; we will derive Theorem \ref{multi} from the hypergraph removal lemma established in \cite{gowers-hyper}, \cite{rodl}, \cite{rs}, \cite{tao-hyper}.

We first observe that Theorem \ref{multi} may be deduced via a lifting trick from the following apparently weaker version:

\begin{theorem}[Multidimensional Szemer\'edi theorem, again]\label{multi-again}  Let $G = (G,+)$ be an additive group, let $m \geq 1$ be integers, and let $A \subset G^m$ be a set with $|A| \geq \delta |G|^m$.  Then there are $\gg_{m,\delta} |G|^{m+1}$ tuples $(a,r) \in G^m \times G$ with the property that
$$ a+re_1, \dots, a+re_m \in A$$
where we adopt the notation that $g(n_1,\dots,n_m) := (n_1 g,\dots,n_m g)$ whenever $g \in G$ and $n_1,\dots,n_m$ are integers, and $e_1,\dots,e_m$ is the standard basis of $\Z^m$.
\end{theorem}

Indeed, to deduce Theorem \ref{multi} from Theorem \ref{multi-again}, let $K := (2k+1)^m$, and let $v_1,\dots,v_K$ be an enumeration of the $K$ $m$-tuples in $\{-k,\dots,k\}^m$.  If $A \subset G^m$, we let $\tilde A \subset G^{m+K}$ be the set
$$ \tilde A := \{ (a,b_1,\dots,b_K) \in G^m \times G^K: a + b_1 v_1 + \dots + b_K v_K \in A \}.$$
If $|A| \geq \delta |G|^m$, then it is clear (by freezing $b_1,\dots,b_K$) that $|\tilde A| \geq \delta |G|^{m+K}$.  Applyig Theorem \ref{multi-again}, we see that there are $\gg_{k,m,\delta}$ tuples $(a,b_1,\dots,b_K,r) \in G^{m+K+1}$ such that
$$ (a,b_1,\dots,b_{i-1},b_i+r,b_{i+1},\dots,b_K) \in \tilde A$$
for all $1 \leq i\leq K$, which by definition of $\tilde A$ implies that
\begin{equation}\label{arv}
a' + r v_i \in A
\end{equation}
for all $i=1,\dots,K$, where $a' := a + b_1 v_1 + \dots + b_K v_K$.  Since each $a' \in G^m$ arises from at most $|G|^K$ tuples $(a,b_1,\dots,b_K)$, we conclude that there are $\gg_{k,m,\delta}$ tuples $(a',r) \in G^{m+1}$ such that \eqref{arv} holds for all $i=1,\dots,K$, and the claim follows.

We now establish Theorem \ref{multi-again}.  Let $G,A,m$ be as in that theorem.  For each $i=1,\dots,m$, we introduce a set $E_i \subset G^{m+1}$, defined as the set of all tuples $(a_1,\dots,a_m,s) \in G^{m+1}$ with the property that
$$ (a_1,\dots,a_{i-1},s-a_1-\dots-a_{i-1}-a_{i+1}-\dots-a_m,a_{i+1},\dots,a_m) \in A.$$
Observe that if $(a_1,\dots,a_m,s)$ lies in the intersection $\bigcap_{i=1}^m E_i$ of all the $E_i$, then by setting $r := s - a_1 - \dots - a_m$, we have $(a_1,\dots,a_m)+ r e_i \in A$ for all $i=1,\dots,m$.  Thus it will suffice to show that
$$
|\bigcap_{i=1}^m E_i| \gg_{m,\delta} |G|^{m+1}.$$
Let $\eps>0$ be a sufficiently small quantity depending on $m,\delta$ to be chosen later.  Suppose for sake of contradiction that
$$
|\bigcap_{i=1}^m E_i| < \eps |G|^{m+1}.$$
 Observe that each $E_i$ is \emph{$i$-invariant} in the sense that the assertion that a given tuple $(a_1,\dots,a_m,s) \in G^{m+1}$ lies in $E_i$ does not depend on the $i^{\operatorname{th}}$ coordinate $a_i$.  Because of this, we may apply the hypergraph removal lemma (see e.g. \cite[Theorem 1.13]{tao-hyper} and conclude (if $\eps$ is small enough depending on $m,\delta$) that there exist $i$-invariant perturbations $E'_i$ of $E_i$ with
\begin{equation}\label{ii-small}
|E'_i \Delta E_i| < \frac{\delta}{m} |G|^{m+1}
\end{equation}
 such that 
\begin{equation}\label{eim}
\bigcap_{i=1}^m E'_i = \emptyset.
\end{equation}
 
We now intersect $E_i$, $E'_i$ with the hyperplane
$$ \Sigma := \{ (a_1,\dots,a_m, a_1+\dots+a_m): a_1,\dots,a_m \in G \}.$$
As this hyperplane sits transversely with respect to the $i$-invariant set $E'_i \Delta E_i$, we conclude from \eqref{ii-small} that
$$ |(E'_i \Delta E_i) \cap \Sigma| < \frac{\delta}{m} |G|^m$$
and hence from the union bound and \eqref{eim}
$$ |\bigcap_{i=1}^m E_i \cap \Sigma| < \delta |G|^m.$$
On the other hand, since $(a_1,\dots,a_m,a_1+\dots+a_m) \in \bigcap_{i=1}^m E_i\cap \Sigma$ whenever $(a_1,\dots,a_m) \in A$, we have
$$ |\bigcap_{i=1}^m E_i \cap \Sigma| \geq |A| \geq \delta |G|^m,$$
giving the desired contradiction.  This completes the proof of Theorem \ref{multi-again} and hence Theorem \ref{multi}.

\end{document}